\documentclass{jhrs} 

\usepackage[all]{xy}
\usepackage{amsmath, amsthm}

\usepackage{latexsym,amsfonts,}

\usepackage{tikz}
\usepackage{mathrsfs}
\usepackage{epsfig}
\usepackage{color}
\usetikzlibrary{arrows}

\makeatletter
\@addtoreset{equation}{section}
\renewcommand{\theequation}{\thesection.\arabic{equation}}
\makeatother
\makeatletter
\@addtoreset{equation}{section}
\renewcommand{\theequation}{\thesection.\arabic{equation}}
\makeatother
\makeatletter
\@addtoreset{equation}{section}
\renewcommand{\theequation}{\thesection.\arabic{equation}}
\makeatother

\newtheorem{prop}{Proposition}[section]
\newtheorem{theo}[prop]{Theorem}
\newtheorem{lem}[prop]{Lemma}
\newtheorem{defi}[prop]{Definition}
\newtheorem{const}[prop]{Constructions}
\newtheorem{coro}[prop]{Corollary}
\newtheorem{obse}[prop]{Observation}

\newtheorem{exam}[prop]{Example}

\newtheorem{leer}[prop]{}
\newtheorem{rema}[prop]{Remark}

\newtheorem{nota}[prop]{Notation}

\def\Top{\mathcal{T}\! op}
\def\top{{\mathcal{T}\! op}^\ast}
\def\topw{{\mathcal{T}\! op}^w}
\def\id{\mathop{\rm id}}
\def\Id{\mathop{\rm Id}}
\def\Mon{\mathcal{M}on}
\def\Monw{\mathcal{M}on^w}

\def\Monc{\mathcal{M}on^{\mathcal{C}}}
\def\Mc{\mathcal{M}^{\mathcal{C}}}
\def\Nc{\mathcal{N}^{\mathcal{C}}}

\def\constH{\textrm{const}^{\mathcal{H}}}
\def\cons{\textrm{const}}
\def\Sgp{\mathop{\rm \mathcal{S}gp}}
\def\P{\mathop{\rm Path}}
\def\hocolim{\mathop{\rm hocolim}\nolimits}
\def\colim{\mathop{\rm colim}\nolimits}
\def\colimH{\mathop{\rm colim}\nolimits^{\mathcal{H}}}

\def\scA{\mathcal{A}}
\def\scB{\mathcal{B}}
\def\scD{\mathcal{D}}
\def\scH{\mathcal{H}}
\def\scM{\mathcal{M}}
\def\scN{\mathcal{N}}
\def\scC{\mathcal{C}}
\def\scJ{\mathcal{J}}
\def\scX{\mathcal{X}}

\def\bL{\mathop{\bf L}}
\def\bR{\mathop{\bf R}}
\def\BH{B^{\mathcal{H}}}
\def\BwH{B^{w{\mathcal{H}}}}
\def\OH{\Omega'^{\mathcal{H}}}
\def\OwH{\Omega'^{w{\mathcal{H}}}}
\def\id{\mathop{\rm id}\nolimits}
\def\Id{\mathop{\rm Id}\nolimits}
\def\Sing{\mathop{\rm Sing}\nolimits}
\def\ob{\mathop{\rm ob}\nolimits}
\def\ev{\mathop{\rm ev}\nolimits}
\def\SSets{\mathcal{SS}ets}
\def\Ho{\mathop{\rm Ho}\nolimits}
\def\op{\mathop{\rm op}\nolimits}

\begin{document}
\title{Homotopy homomorphisms and the classifying space functor}

\author{R.M. Vogt}             
\email{rvogt@uos.de}       
\address{Fachbereich Mathematik/Informatik\\ 
         Universit\"at Osnabr\"uck\\
         Germany\\
         D-49069 Osnabr\"uck\\
         }


\begin{abstract}
We show that the classifying space functor $B:\mathcal{M}on \to \top$ from the category
of topological monoids to the category of based spaces is left adjoint to 
the Moore loop space functor $\Omega':\top\to\Mon$ after we have localized $\Mon$ with respect to all
homomorphisms whose underlying maps are homotopy equivalences and $\top$ with respect
to all based maps which are (not necessarily based) homotopy equivalences. It is well-known that
this localization of $\top$ exists, and we show that the localization of $\Mon$ is the category
of monoids and homotopy classes of homotopy homomorphisms. To make this statement precise 
 we have to modify the classical definition of a homotopy homomorphism, and we discuss the necessary changes.
 The adjunction is induced by an
adjunction up to homotopy  $B:\scH\Monw\leftrightarrows \topw:\Omega'$ between the category of
 well-pointed monoids and  homotopy homomorphisms and the category of well-pointed spaces. This
adjunction is shown to lift to diagrams. As a consequence, the well-known derived adjunction
between the homotopy colimit and the constant diagram functor can also be seen to be induced 
by an adjunction up to homotopy before taking homotopy classes. As applications we among other
things deduce a more algebraic version of the group completion theorem and show
 that the classifying space functor preserves homotopy colimits up to natural
homotopy equivalences. 
\end{abstract}

\received{Month Day, Year}   
\revised{Month Day, Year}    
\published{Month Day, Year}  
\submitted{Name of Editor}  

\volumeyear{2012} 
\volumenumber{1}  
\issuenumber{1}   

\startpage{1}     

\maketitle

\section{Introduction}
Let $\Top$ denote the category of $k$-spaces, $\top$ the category of based $k$-spaces, and $\topw$ the
category of well-pointed $k$-spaces. Recall that a space $X$ is a \textit{$k$-space} if $A\subset X$ is closed iff $p^{-1}(A)$ is 
closed in $C$ for each map $p:C\to X$ where $C$ is a compact Hausdorff space, and that a space is called \textit{well-pointed}
if the inclusion of the base point is a closed cofibration. 

Let $\Mon$ denote the category of topological monoids and continuous homomorphisms,
and $\Monw$ and $\mathcal{C}\Mon$ 
the full subcategories of well-pointed, respectively, commutative monoids. A monoid is canonically based by its unit.

We are interested in the relationship between Milgram's classifying space functor
$B: \Mon \to \top$ and the Moore loop space functor $\Omega ':\top \to \Mon$ (for explicit definitions see Section 4).

The related question for commutative monoids is easily answered:
it is well-known that the classifying space $BM$ of a commutative monoid is a commutative monoid \cite{Milgram}, so that we
have a functor $B:\mathcal{C}\Mon\to \mathcal{C}\Mon$. The usual loop space functor induces a functor $\Omega: 
\mathcal{C}\Mon\to \mathcal{C}\Mon$ by defining the multiplication in $\Omega M$ by point-wise multiplication in $M$.
The category $\mathcal{C}\Mon$ is enriched over $\top$ in an obvious way, and it is tensored and cotensored
(for definitions see \cite{Borceux} or Section 3). The cotensor $M^K$ of $M\in \mathcal{C}\Mon$ and $K\in \top$ is the
function space with point-wise multiplication. It is well-known that $B(M)\cong M\boxtimes S^1$, the tensor of
$M$ and $S^1$. Since $-\boxtimes K$ is left adjoint to $(-)^K$  we obtain:

\begin{prop}\label{1_1}
The functors
$$
B:\scC\Mon\rightleftarrows\scC\Mon :\Omega
$$
 form a $\top$-enriched adjoint pair.
\end{prop}

In the non-commutative case there is no hope for a similar result. A candidate for a right adjoint of
the classifying functor
$$
B: \Mon \to \top$$
is the Moore loop space functor
$$
\Omega':\top \to \Mon,
$$
but $\Omega'$ does not preserve products. In fact, there is no product preserving functor
$$F:\top \to \Mon$$
such that $F(X)\simeq \Omega(X)$ for all $X$ \cite[Prop. 6.1]{BV}.

\begin{rema}\label{1_1a}
In \cite{Fied} Fiedorowicz showed that the Moore loop space functor into a different target category is right
adjoint to what he called the Moore suspension functor: Let $\top[\mathbb{R}_+]$ be the category whose objects are based spaces $X$ together with a continuous map
$p:X\to \mathbb{R}_+$ (the non-negative real numbers) such that $p^{-1}(0)=\ast$ and whose morphisms are maps over
$\mathbb{R}_+$. Then 
$$\Omega':\top\to \top[\mathbb{R}_+]\qquad X\mapsto (\Omega'X,l),$$
where $l$ is the length function, has this Moore suspension functor as left adjoint.
\end{rema} 

The  Moore loop space funtor $\Omega':\top\to \Mon$ preserves products up to natural homotopy. 
So one might expect it to be a right
adjoint of $B$ after formally inverting homotopy equivalences. We will prove this in this paper.

We will have to localize our categories $\scC$, and it is a priori not clear that these localizations exist. 
A common procedure
is to define a Quillen model structure on $\scC$ such that the morphisms we want to invert are the
weak equivalences in these structures. The localization then is the homotopy category $\Ho\scC$ associated
with this model structure.

There are two standard model structures on $\Top$: The structure 
due to Quillen \cite{Qu} whose weak equivalences are weak homotopy equivalences and whose fibrations are 
Serre fibrations, and the structure due to Str{\o}m \cite{Strom2} whose weak equivalences are homotopy
equivalences, whose fibrations are Hurewicz fibrations, and whose cofibrations are closed cofibrations.

Although mainstream homotopy theory usually works with the Quillen model structure and the proofs of our results
would be considerably shorter in this context (because we could use the rich literature, in particular, the 
results of Fiedorowicz \cite{Fied}), we choose the Str{\o}m setting because we share 
D. Puppe's point of view \cite{Puppe}: 
``Frequently a weak
homotopy equivalence is considered as good as a genuine one, because for spaces having
the homotopy type of a $CW$-complex there is no difference and most interesting 
spaces in algebraic topology are of that kind. I am not going to argue against this 
because I agree with it, but I do think that the methods by which we establish
the genuine homotopy equivalences give some new insight into homotopy theory.''
Moreover, there are spaces of interest which rarely have the homotopy type of a $CW$ complex such
as function spaces and spaces of foliations, which account for a growing interest in results
in the Str{\o}m setting.

So we call a based map in $\top$ a weak equivalence if it is a not necessarily based homotopy equivalence,
and a homomorphism in $\Mon$ a weak equivalence if the underlying map of spaces is a weak equivalence in
$\top$. 
Let $\Ho\top$ and $\Ho\Mon$ be the categories obtained from $\top$ respectively $\Mon$ by formally
inverting weak equivalences. 

\begin{theo}\label{1_4}
 The categories $\Ho\top$ and $\Ho\Mon$ exist and the classifying space functor and the Moore loop space 
functor induce a derived adjoint pair
$$
\Ho B:\Ho\Mon\rightleftarrows \Ho\top:\Ho\Omega'
$$
\end{theo}

\begin{rema}\label{1_2}
 This contrasts the situation in the simplicial category: The loop group functor $G: \SSets \to
\mathcal{SG}roups$ from simplicial sets to simplicial groups is left adjoint to the simplicial
classifying space functor $\overline{W}: \mathcal{SG}roups \to \SSets$ (e.g. see \cite[Lemma V.5.3]{GJ}).
\end{rema}

With our choice of weak equivalences the Str{\o}m model structure on $\Top$ lifts to $\top$ so that
$\Ho\top$ exists, but in contrast to the Quillen model structure, it is not known that
 the Str{\o}m model lifts to $\Mon$
(there is a model structure on $\Mon$ whose weak equivalences
are homotopy equivalences in $\Mon$ rather than homotopy equivalences of underlying spaces; this follows from
work of Cole \cite{Cole} and Barthel and Riel \cite{BaRi}).

In the construction of $\Ho\Mon$ in the Str{\o}m setting homotopy homomorphisms between monoids come into play: 
A topological monoid can be considered as an algebra over the operad $\scA ss$
of monoid structures or as a topologically enriched 
category with one object. The homotopy homomorphisms of this paper are based on the enriched category aspect and
describe ``functors up to coherent homotopies''.
They were
 introduced for monoids by Sugawara in 1960 \cite{Sug} and extensively
studied by Fuchs in 1965 \cite{Fuchs}. Homotopy homomophisms of $\scA ss$-algebras were introduced in
\cite{BV0}, and we will indicate their relation to the ones considered in this paper in Section 2.
An extension of our results to arbitrary category objects in $\Top$ may be of separate interest.

 If we define a semigroup to be a topological space with a continuous associative multiplication,
an inspection of the definition shows
that a homotopy homomorphism $f:M\to N$ of monoids is nothing but a semigroup homomorphism $\overline{W}M\to N$ where $\overline{W}$ is a variant
 of the Boardman-Vogt $W$-construction \cite{BV0} (not to be confused with the functor $\overline{W}$
of Remark \ref{1_2}). If $\Sgp$ denotes the category of semigroups and continuous homomorphisms then $\overline{W}:\Sgp\to \Sgp$ is a functor
equipped with a natural transformation $\overline{\varepsilon}:\overline{W}\to \Id$. The Boardman-Vogt
 $W$-construction $W:\Mon\to \Mon$ and its associated
 natural transformation $\varepsilon : W\to \Id$ are obtained
from $(\overline{W},\  \overline{\varepsilon})$ by factoring out a unit relation. In particular, for 
any monoid $M$ there is a natural projection $\varepsilon'(M):
\overline{W}M\to WM$ of semigroups such that $\varepsilon(M)\circ \varepsilon'(M)=\overline{\varepsilon}(M)$.

The lack of conditions for the unit is
 an indication that Sugawara's notion of a homotopy homomorphism is not quite the correct one. So we define 
unitary  homotopy homomorphisms from $M$ to $N$ to be monoid homomorphisms $WM\to N$; those were studied in 1999 by Brinkmeier \cite{Brink}. 

Composition of homotopy homomorphisms and their unitary versions is only associative up to homotopy. To obtain genuine categories of monoids and (unitary) homotopy homomorphisms
 we modify both notions: A  homotopy homomorphisms from $M$ to $N$ will be a semigroup homomorphism $\overline{W}M\to \overline{W}N$
and a unitary one a monoid homomorphism $WM\to WN$.
>From a homotopy theoretical point of view this modification is not significant: 

\begin{prop}\label{4_a}
If $M,\ N$ are monoids and $M$ is well-pointed and $G,\ H$ are semigroups then the maps
$$\begin{array}{rcl}
\varepsilon(N)_\ast:\Mon (WM,WN)& \to &\Mon(WM,N) \\ 
\overline{\varepsilon}(N)_\ast:\Sgp(\overline{W}G,\overline{W}H)& \to & 
\Sgp(\overline{W}G,H)
\end{array}
$$ 
are homotopy equivalences. 
\end{prop}

It is well-known that $WM\to M$ has the flavor of a cofibrant replacement of $M$ as known from model category theory provided $M$
is well-pointed (e.g. see \cite{BM}, \cite{Vogt2}). So it is no surprise that the category of well-pointed monoids and homotopy classes of unitary homotopy homomorphisms is the
localization of $\Monw$ with respect to its weak equivalences. If we want to construct $\Ho \Mon$ we have to relax unitary homotopy homomorphisms
to homotopy unitary homotopy homomorphisms and the corresponding statement holds. We will study these various notions of homotopy homomorphisms
in Section 2 in detail.

The lack of the appropriate Quillen model structure in some of our categories
 is made up for by their topological enrichment with nice properties. This topological
enrichment allows us to prove stronger results. E.g. the restriction of Theorem \ref{1_4} to the well-pointed case
is the path-component version of the following result.

\begin{theo}\label{1_5}
 Let $\scH \Monw$ be the category of well-pointed monoids and unitary homotopy homomorphisms. Then the classifying space functor and the
Moore loop space functor induce an adjunction up to homotopy
$$\scH \Monw \leftrightarrows \topw.$$
\end{theo}

In Section 3 we will introduce the necessary notions to make this precise. There we will also recall basic facts from enriched category
theory and show that topologically enriched categories with a class of weak equivalences which admit a cofibrant replacement functor
can be localized. We believe that these results are of separate interest.

In Section 4 we prove Theorem \ref{1_5} and related results and hence Theorem \ref{1_4}. In Section 5 we draw some immediate
consequences of Theorem \ref{1_4} and of the intermediate steps in the proof of Theorem \ref{1_5}.

E.g. we obtain yet another but considerably shorter proof of a strong
version of the James construction.

\begin{defi}\label{1_6a}
 A \textit{Dold space} is a topological space admitting a numerable cover $\{ U_\gamma;\ \gamma \in \Gamma\}$ such that each
inclusion $U_\gamma \subset X$ is nullhomotopic.
\end{defi}
A space of the homotopy type of a $CW$-complex is a Dold space. For more details on Dold spaces see \cite{Schwam}.

\begin{prop}\label{1_6} (1) If $X$ is a well-pointed space and $JX$ is the based free topological
 monoid on $X$ (the James construction), then $BJX\simeq \Sigma X$.\\
(2) If $X$ is a well-pointed path-connected Dold space, then $JX\simeq \Omega\Sigma X$.
\end{prop}

Part (2) was first proven in \cite{DKP}, shorter proofs can be found in \cite{Puppe} and \cite{Schwam}.

We also obtain a new interpretation of the group completion theorem of a
monoid without any additional assumptions on the multiplication.

\begin{defi}\label{1_7}
 A topological monoid is called \textit{grouplike} if it admits a continuous homotopy inversion.
\end{defi}
A standard example of a grouplike monoid is the Moore loop space $\Omega'X$ of a space $X$.

\begin{theo}\label{1_8}
 Let $M$ be a well-pointed topological monoid. Then there is a unitary homotopy homomorphism
$\mu_M:M\to \Omega'BM$, natural up to homotopy, having the following universal property:
Given any unitary homotopy homomorphism $f:M\to N$ into a grouplike monoid $N$ there is a
unitary homotopy homomorphism $\bar{f}:\Omega'BM\to N$, unique up to homotopy, such that
$\bar{f}\circ \mu_M\simeq f$. (Here homotopy means homotopy in the category, i.e.
homotopy through unitary homotopy homomorphisms.)
\end{theo}

>From the intermediate steps of the proof of Theorem \ref{1_5} we obtain the following
extension and strengthening of a theorem of Fuchs \cite[Satz 7.7]{Fuchs}

\begin{prop}\label{1_9}
 (1) If $M$ and $N$ are well-pointed monoids and $N$ is group\-like then 
$$B:\Mon(WM,WN)\to \top(BWM,BWN)$$
is a homotopy equivalence.\\
(2) If $X$ is a well-pointed path-connected Dold space
then $W\Omega': \topw(X,Y)\to\Monw(W\Omega'X,W\Omega'Y)$ is a homotopy equivalence.
\end{prop}

The reader may object that Fuchs considers homotopy homomorphisms while Proposition \ref{1_9}
addresses unitary homotopy homomorphisms. Since Fuchs only considers well-pointed grouplike monoids 
and all his spaces are of the homotopy type of $CW$-complexes the two
notions are linked by 

\begin{prop}\label{1_10}
 Let $M$ and $N$ be well-pointed monoids and $N$ be group\-like. Then 
$$(\varepsilon')^\ast: \Mon(WM,N)\to \Sgp(\overline{W}M,N)$$
is a homotopy equivalence,
\end{prop}

Section 6 deals with diagrams in topologically enriched categories $\scM$ with weak equivalences and a ``good''
cofibrant replacement functor. We first show that their localizations with respect to maps of diagrams
which are objectwise weak equivalences exist. We then show that the well-known derived adjunction
induced by the colimit functor and the constant diagram functor is the path-component version of 
an adjunction up to homotopy between the homotopy colimit functor and the constant diagram functor.
We believe that this is of separate interest, too. We then show that the 
 homotopy adjunction of Theorem \ref{1_5}
lifts to a homotopy adjunction between the corresponding categories of diagrams. In contrast to
strict adjunctions this is a priori not clear, because the associated unit is natural only up to
homotopy and hence does not lift to diagrams. We apply this result to prove 

\begin{theo}\label{1_11}
 The classifying space functor $B:\Mon\to \top$ preserves homotopy 
colimits up to natural homotopy equivalences.
\end{theo}

The path-component versions of most of our main results are more or less known if we restrict to grouplike monoids. The
paper extends these results to general monoids and shows that they arise from stronger statements. Moreover, we show 
that a topological enrichment with good properties can make up for the non-existence of Quillen model structures.

\textit{Acknowledgement:} I want to thank P. May for pointing out possible shortcuts to Theorem \ref{1_4}
in the Quillen context and
for an extended e-mail exchange on the presentation of the paper, and to M. Stelzer for clarifying discussions. I am
indebted to the referee for his careful reading of the paper, for requiring a number of clarifications, 
for suggesting explicit improvements of a number 
of formulations which had been a bit opaque, and for his patience with my many typos. In particular, the organisation
of the present proof of Proposition \ref{4_8} is due to him.

\section{Homotopy homomorphisms revisited}
In 1960 Sugawara introduced the notion of a strongly homotopy multiplicative map between monoids, which 
we will call a homotopy
 homomorphism or $h$-morphism, for short \cite{Sug}.

\begin{defi}\label{2_1}
A \textit{homotopy homomorphism}, or $h$-\textit{morphism} $f:M\to N$ between two monoids is a sequence of maps
$$
f_n: M^{n+1}\times I^n\longrightarrow N\quad n\in\mathbb{N}
$$
such that $(x_i\in M, t_j\in I)$
\begin{multline*}
f_n(x_0, t_1, x_1,t_2,\ldots,t_n,x_n)\\
= \left\{
\begin{array}{ll}
f_{n-1}(x_0,t_1,\ldots, x_{i-1}\cdot x_i, \ldots,t_n,x_n) & \textrm{ if } t_i=0\\
f_{i-1}(x_0,t_1,\ldots, x_{i-1})\cdot f_{n-i}(x_i,t_{i+1},\ldots,x_n) & \textrm{ if } t_i=1.
\end{array}
\right.
\end{multline*}
We call $f_0:M\to N$ the \textit{underlying map} of $f$.

If in addition $f_0(e_M)=e_N$ and
\begin{multline*}
f_n(x_0, t_1, x_1,t_2,\ldots,t_n,x_n)\\
= \left\{
\begin{array}{ll}
f_{n-1}(x_1,t_2, \ldots,x_n) & \textrm{ if } x_0=e_M\\
f_{n-1}(x_0,\ldots, x_{i-1},\max(t_i, t_{i+1}),x_{i+1},\ldots, x_n)& \textrm{ if } x_i=e_M\\
f_{n-1}(x_0, t_1,\ldots,x_{n-1}) & \textrm{ if } x_n=e_M
\end{array}
\right.
\end{multline*}
where $e_M\in M$ and $e_N\in N$ are the units. We call $f$ a \textit{unitary homotopy homomorphism} or $uh$-\textit{morphism}, for short.
\end{defi}

Since an $h$-morphism does not pay tribute to the unit it does not seem to be the right notion for maps between monoids. E.g. if 
we require $f_0$ to be a based map so that it preserves the unit we would like the path
$$
\xymatrix{
f_0(x_0\cdot x_1)\ar@{-}[rrrr]^{f_1(x_0,t,x_1)} &&&& f_0(x_0)\cdot f_0(x_1)
}
$$
to be the constant one, if $x_0$ or $x_1$ is the unit. Unitary $h$-morphisms have this 
property. Nevertheless, in the past 
one usually considered $h$-morphisms because the additional conditions for $uh$-morphisms make it harder to work with them.

We will later find it more convenient to work with homotopy unitary homotopy homomorphisms which preserve the unit only up
to homotopy. We will introduce those at the end of this section.

The most extensive study of $h$-morphisms and their induced maps on classifying spaces was done by Fuchs \cite{Fuchs}, who  
constructed composites of $h$-morphisms, proved that composition is homotopy associative and stated that an $h$-mor\-phism 
$f:M\to N$ whose underlying map is a homotopy equivalence has a homotopy inverse $h$-morphism $g:N\to M$. In fact, he constructed 
$g_0$, $g_1$ and the homotopies $g\circ f\simeq\id$ and $f\circ g\simeq\id$ in dimensions 0 and 1 in \cite[p.205-p.208]{Fuchs}, but left the rest to the reader.
He produced a complete proof in \cite{Fuchs2}.

We handle these problems by interpreting homotopy homomorphisms as genuine homomorphisms of a ``cofibrant'' replacement of $M$.

By a \textit{semigroup} we will mean a $k$-space with a continuous associative multiplication. Let $\Sgp$ denote the category of semigroups and continuous homomorphisms.

\begin{const}\label{2_2}
We will construct continuous functors
$$
\overline{W}:\Sgp\longrightarrow\Sgp \quad \textrm{ and }\quad W:\Mon\longrightarrow\Mon
$$
and natural transformations
$$
\overline{\varepsilon}: \overline{W}\longrightarrow\Id \quad\textrm{ and }\quad \varepsilon:W\longrightarrow\Id 
$$
as follows:
$$
\overline{W}M=\left(\coprod\limits^\infty_{n=0} M^{n+1}\times I^n\right)/ \sim
$$
with the relation

(1) $(x_0,t_1,x_1,t_2,\ldots,t_n,x_n)\sim(x_0,t_1,\ldots,x_{i-1}\cdot x_i,\ldots,t_n,x_n)$\qquad if $t_i=0$\\
and $WM$ is the quotient of $\overline{W}M$ by imposing the additional relations

(2)  $(x_0, t_1, x_1,t_2,\ldots,t_n,x_n)\\
{}\qquad \ \sim \left\{
\begin{array}{ll}
(x_1,t_2, \ldots,x_n) & \textrm{ if } x_0=e\\
(x_0,\ldots, x_{i-1},\max(t_i, t_{i+1}),x_{i+1},\ldots, x_n)& \textrm{ if } x_i=e\\
(x_0, t_1,\ldots,x_{n-1}) & \textrm{ if } x_n=e
\end{array}
\right.
$\\
The multiplications of $\overline{W}M$ and $WM$ are given on representatives by 
$$
(x_0,t_1,\ldots,x_k)\cdot(y_0,u_1,\ldots,y_l)=(x_0, t_1,\ldots x_k, 1, y_0, u_1,\ldots,y_l).
$$
The natural transformations $\overline{\varepsilon}$ and $\varepsilon$ are defined by
$$
\overline{\varepsilon}(M),\ \varepsilon(M):(x_0,t_1,\ldots,x_n)\longmapsto x_0\cdot x_1\cdot\ldots\cdot x_n.
$$
Their underlying maps have natural sections
$$
\bar{\iota}(M),\ \iota(M): x \longmapsto (x)
$$
which are not homomorphisms, and there is a homotopy over $M$
$$
h_s:(x_0,t_1,x_1,\ldots,t_n,x_n)\longmapsto(x_0,s\cdot t_1,x_1,\ldots,s\cdot t_n,x_n)
$$
from $\overline{\iota}(M)\circ\overline{\varepsilon}(M)$ respectively $\iota(M)\circ\varepsilon(M)$ to the identity. 
In particular, $\overline{\varepsilon}(M)$ and $\varepsilon(M)$ are shrinkable as maps.

If $M$ is a monoid the projection 
$$\varepsilon'(M):\overline{W}M\to WM$$

is a homomorphism of semigroups satisfying
$$
\bar{\varepsilon}(M)=\varepsilon(M)\circ\varepsilon'(M)\qquad \textrm{and}\qquad \varepsilon'(M)\circ \overline{\iota}(M) = \iota(M).
$$
\end{const}

By inspection we see

\begin{obse}\label{2_3}
\begin{enumerate}
\item[(1)]
$h$-morphisms $(f_n): M\to N$ correspond bijectively to homomorphisms $\bar{f}:\overline{W}M\to N$ of semigroups, and $f_0=\bar{f}\circ\overline{\iota}(M)$
\item[(2)]
$uh$-morphisms $(f_n): M\to N$ correspond bijectively to homomorphisms $f: WM\to N$ of monoids, and $f_0=f\circ\iota(M)$
\end{enumerate}
\end{obse}

\begin{obse}\label{2_3a}
Algebraically, $\overline{W}M$ is a free semigroup and $WM$ is a free monoid. The indecomposables are precisely those elements which have a representative
$(x_0,t_1,x_1,\ldots,x_n)$ where no $t_i$ equals $1$.
\end{obse}

\begin{leer}\label{2_3b} \textbf{The formal relation between $\overline{W}$ and $W$:}
 The forgetful functor $i:\Mon\to\Sgp$ has a left adjoint 
$$(-)_+:\Sgp\to\Mon,\qquad G\mapsto G_+,$$
where $G_+=G\sqcup \{\ast\}$ with $\ast$ as unit. It follows from the definitions that the diagram
$$
\xymatrix{
\Sgp \ar[d]_{(-)_+}\ar[r]^{\overline{W}} & \Sgp \ar[d]^{(-)_+}\\
\Mon\ar[r]^W &\Mon
}
$$
commutes up to natural isomorphisms in $\Mon$.
\end{leer}

Both constructions have a universal property, which is a consequence of the following result.
We give $\top(X,Y)$ and $\Top (X,Y)$ the $k$-function space topology, obtained by turning the space of all maps from $X$ to $Y$ with the compact-open topology 
into a $k$-space. We give $\Mon(M,N)$ and $\Sgp(M,N)$ the subspace topologies of the corresponding function spaces in $\top$ respectively $\Top$.

\begin{defi}\label{2_3c}
 We call a homomorphism $f:M\to N$ in $\Mon$ or $\Sgp$ a \textit{weak equivalence} if its underlying map of spaces is a homotopy equivalence in $\Top$.
(Recall that a weak equivalence in $\Mon$ is a homotopy equivalence of underlying spaces in $\top$ if $M$ and $N$ are well-pointed.)
\end{defi}

\begin{prop}\label{2_4}
\begin{enumerate}
\item[(1)]Let $M$ be a well-pointed monoid and $p:X\to Y$ a homomorphism of monoids. Let
$$
p_\ast:\Mon(WM,X)\longrightarrow\Mon(WM,Y)
$$
be the induced map.
If $p$ is a fibration of underlying spaces, so is $p_\ast$. If $p$ is a weak equivalence, $p_\ast$ is a homotopy equivalence.
\item[(2)]
The same holds for $\overline{W}$ and an arbitrary object $M$ in the category $\Sgp$.
\end{enumerate}
\end{prop}

\begin{proof}
 Let $p:X\to Y$ be a weak equivalence. By the HELP-Lemma \cite{Vogt} in $\Top$ with the Str{\o}m
model structure \cite{Strom} we have to show: Given a diagram of spaces

(A)
$$ \xymatrix
{
 A\ar[d]^i\ar[r]^(0.25){\bar{f}_A} &\Mon(WM,X)\ar[d]^{p_\ast}\\
B\ar[r]^(0.25){\bar{g}} &\Mon(WM,Y)
}
$$
which commutes up to a homotopy $\bar{h}_{A,t}:\bar{g}\circ i\simeq p_\ast\circ \bar{f}_A$, where $i$ is a closed cofibration,
there are extensions $\bar{f}:B\to \Mon(WM,X)$ of $\bar{f}_A$ and $\bar{h}_t:B\to \Mon(WM,Y)$ of $\bar{h}_{A,t}$ such that
$\bar{h}_t: \bar{g}\simeq p_\ast\circ \bar{f}$.

Passing to adjoints we obtain a diagram
$$ \xymatrix
{
WM \times A\ar[d]^{id\times i}\ar[r]^(0.65){f_A} & X\ar[d]^p\\
WM \times B \ar[r]^(0.65)g& Y
}
$$
commuting up to a homotopy $h_{A,t}$, such that each $f_a=f_A|WM\times \{a\}$, each $g_b=g|WM\times \{b\}$, and each
$h_{a,t}=h_{A,t}|WM\times \{a\}$ is a homomorphism. We have to construct extensions $f:WM\times B\to X$ and $h_t:WM\times B\to Y$
of $f_A$ and $h_{A,t}$ such that $h_t:g\simeq p\circ f$ and each $h_{b,t}$ and $f_b, \ b\in B$ is a homomorphism.

We filter $WM\times B$ by closed subspaces $F_n\times B$, where $F_n$ is the submonoid of $WM$ generated by all 
elements having a representative $(x_0,t_1,\ldots,t_k,x_k)$ with $k\leq n$. We put $F_{-1}=\{e\}$. Then $f$ and $h_t$ are
uniquely determined on $F_{-1}\times B$. 

Suppose that $f$ and $h_t$ have been defined on $F_{n-1}\times B$. An element $(x_0,t_1,\ldots,t_n,x_n)$ represents
an element in $F_{n-1}$ iff one of the following conditions holds
\begin{itemize}
 \item some $x_i=e$\qquad (relation \ref{2_2}.2)
\item some $t_i=0$\qquad (relation \ref{2_2}.1)
\item some $t_i=1$\qquad (it represents a product in $F_{n-1}$).
\end{itemize}
If $DM^{n+1}\subset M^{n+1}$ denotes the subspace of points with some coordinate $e$, then $f$ and $h_t$ are already
defined on $(DM^{n+1}\times I^n\cup M^{n+1}\times\partial I^n)\times B\cup M^{n+1}\times I^n\times A$. The elements in
$(M^{n+1}\times I^n)\backslash (DM^{n+1}\times I^n\cup M^{n+1}\times\partial I^n)$ represent indecomposables of filtration $n$, but not of
lower filtration.  Consider the diagram

(B)
$$\xymatrix
{
(DM^{n+1}\times I^n\cup M^{n+1}\times\partial I^n)\times B\cup M^{n+1}\times I^n\times A\ar[r]^(0.88)f\ar[d]^j & X\ar[d]^p\\
M^{n+1}\times I^n\times B\ar[r]^g & Y
}
$$
(in abuse of notation we use $g$ for the composite $M^{n+1}\times I^n\times B\to WM\times B\to Y$). Diagram (B) commutes up to the
homotopy $h_t$ and we need an extension of $f$ and $h_t$ to $M^{n+1}\times I^n\times B$. These extensions exist by the HELP-Lemma,
because our assumptions ensure that $j$ is a closed cofibration.
So we have defined $f$ and $h_t$ for indecomposable generators $(x_0,t_1,\ldots,t_n,x_n)$ of $F_n$. We extend these maps to $F_n\times B$
by the conditions that each $f_b$ and $h_{b,t},\ b\in B$ be a homomorphism using Observation \ref{2_3a}.

Now suppose that $p$ is a fibration. By \cite[Thm. 8]{Strom} we need to consider a commutative diagram (A), where $i$ is a closed
cofibration and a homotopy equivalence, and we have to find an extension $\bar{f}: B\to \Mon(WM,X)$ of $\bar{f}_A$ such that 
$\bar{g}=p_\ast\circ \bar{f}$. We proceed as above. In the inductive step we have a commutative diagram (B). Since $i$ is a closed
cofibration and a homotopy equivalence so is $j$ by the pushout-product theorem for cofibrations. Hence the required extension
$f:M^{n+1}\times I^n\times B\to X$ exists by \cite[Thm. 8]{Strom}.

Part (2) is proved in the same way starting with $F_{-1}M=\emptyset$.
\end{proof}

As an immediate consequence we obtain the

\begin{leer}\label{2_5}
\textbf{Lifting Theorem:} (1) Given homomorphisms of monoids
$$
\xymatrix{
&& X \ar[d]^p\\
WM \ar[rr]^f && Y
}
$$
such that $p$ is a weak equivalence and $M$ is well-pointed, then there exists a homomorphism $g:WM\to X$, unique up to homotopy in $\Mon$ (i.e. a homotopy through
homomorphisms), such that $f\simeq p\circ g$ in $\Mon$. \\
If, in addition, the underlying map of $p$ is a fibration there is a homomorphism $g: WM\to X$, unique up to homotopy in $\Mon$, such that $f=p\circ g$.

(2) For $\overline{W}$ the analogous results hold in the category $\Sgp$.
\end{leer}

\begin{leer}\label{2_6a}
By Proposition \ref{2_4} the second one of the maps 
$$\begin{array}{rcl}
\varepsilon(N)_\ast:\Mon (WM,WN)& \to &\Mon(WM,N) \\ 
\overline{\varepsilon}(N)_\ast:\Sgp(\overline{W}M,\overline{W}N)& \to & 
\Sgp(\overline{W}M,N)
\end{array}
$$ 
is a homotopy equivalence, and the first one is a homotopy equivalence if $M$ is well-pointed. 
\end{leer}

To guarantee the well-pointedness condition we introduce the whiskering functor.

\begin{leer}\label{2_6b}
 \textbf{The whiskering construction:}
We define a functor
$$
V^t:\top\to\topw
$$ by $V^t(X,x_0)=(X\sqcup I)/(x_0\sim 1)$
and choose $0\in I$ as base-point of $X_I$. Then
$V^tX$ is well-pointed, and
the natural map $q(X): V^tX\to X$ mapping $I$ to $x_0$ is a homotopy equivalence. Its homotopy inverse
$\bar{q}(X): X\to V^tX$
is the canonical map. 
If $X$ is well-pointed, $q(X)$ is a based homotopy equivalence.

This functor lifts to a functor
$$
V:\Mon\to \Monw$$
defined by $V(M)=V^t(M)$ with $x_0$ replaced by $e_M$
with the multiplication
$$
x\cdot y=\left\{
\begin{array}{ll}
x\cdot y\in M & \textrm{if } x,y\in M\\
x & \textrm{if } x\in M, \ y\in I\\
y & \textrm{if } y\in M, \ x\in I\\
\textrm{max}(x,y)& \textrm{if }  x,y \in I
\end{array}
\right.
$$
Since $0\in I$ is the unit of $VM$ the monoid $VM$ is well-pointed. The
natural map $q(M): VM\to M$ is a weak equivalence in $\Mon$, but observe that $\bar{q}(M): X\to VM$
is not a homomorphism because it does not preserve the unit.
\end{leer}

A homomorphism $f:WVM\to N$ can be considered a homotopy unitary homotopy homomorphism. Strictly speaking,
the underlying map of $f:WVM\to N$ is
 $$f_0=f\circ \iota(VM)\circ \bar{q}(M):M\to VM\to WVM\to N.$$
 We note that $f_0$ preserves the unit $e_M$ only up to homotopy.

By \ref{2_6a} the following \textbf{change of our notations of homotopy homomorphisms}
is insignificant from a homotopy theoretic point of view:

\begin{defi}\label{2_6} \textbf{From now on}
a \textit{homotopy unitary homotopy homomorphism}, $huh$-morphism for short, 
from $M$ to $N$ is a homomorphism $f:WVM\to WVN$. Its
\textit{underlying map} is $q(N)\circ\varepsilon(VN)\circ f\circ\iota(VM)\circ \bar{q}(M)$. \\
A \textit{unitary homotopy homomorphism}, $uh$-morphism for short, from $M$ to $N$ is a homomorphism $f:WM\to WN$. Its 
\textit{underlying map} is $\varepsilon(N)\circ f\circ\iota(M)$. \\
A \textit{homotopy homomorphism}, $h$-morphism for short, from the semigroup $M$ 
to the semigroup $N$ is a homomorphism $f:\overline{W}M\to\overline{W}N$.
Its \textit{underlying map} is $\overline{\varepsilon}(N)\circ f\circ\overline{\iota}(M)$.
\end{defi}

This solves the problem of composition, and from \ref{2_4} we obtain

\begin{prop}\label{2_7}
If $f:WM\to WN$ is a $uh$-morphism from $M$ to $N$ whose underlying map is a homotopy equivalence, and $M$ and $N$ are well-pointed,
then $f$ is a homotopy equivalence in the category $\Mon$.\\
If $f:WVM\to WVN$ is a $huh$-morphism from $M$ to $N$, whose underlying map is a homotopy equivalence, 
then $f$ is a homotopy equivalence in the category $\Mon$.\\
 The analogous statement in $\Sgp$ holds for homomorphisms $\overline{W}M\to\overline{W}N$.
\end{prop}

Monoids are algebras over the operad $\scA ss$ of monoid structures, and there is the notion of an ``operadic'' homotopy homomorphism defined
by Boardman and Vogt in \cite{BV0}. M. Klioutch compared the operadic notion with the one considered in this paper and could show
\cite{Kli}

\begin{prop}\label{2_8}
Let $M$ and $N$ be well-pointed monoids and let H$(M,N)$ be the space of operadic homotopy homomorphisms from $M$ to $N$, then
there is a natural homotopy equivalence
$$\textrm{H}(M,N)\simeq \Mon(WM,N).$$
\end{prop}

\section{Categorical prerequisites and localizations}

The functors $WV:\Mon\to\Mon$ and $\overline{W}:\Sgp\to \Sgp$ resemble cofibrant replacement functors
as known from Quillen model category theory. Unfortunately, there is no known model category structure
on $\Mon$ with our choice of weak equivalences. This draw-back is made up by the topological
enrichment of our categories as we will see in this section.

Our categories are enriched over $\top$ or $\Top$. So we have a natural notion of homotopy.
Moreover, they are \textit{tensored} and \textit{cotensored}. Recall that a $\top$-enriched
category $\scM$ is tensored and cotensored (over $\top$) if there are functors
$$
\begin{array}{rclcl}
\top\times\scM& \to & \scM , & & (X,M) \mapsto X\boxtimes M\\
(\top)^{\op} \times\scM &\to&\scM,  & & (X,M) \mapsto M^X
\end{array}
$$
and natural homeomorphisms
$$
\scM(X\boxtimes M, N)\cong \top(X,\scM(M,N))\cong \scM (M,N^X).
$$
These properties imply that for based spaces $X$ and $Y$ and objects $M\in \scM$ there are natural
isomorphisms
$$(X\wedge Y)\boxtimes M\cong X\boxtimes (Y\boxtimes M).$$

The definition in the $\Top$-enriched case is similar. To distinguish between the based and the non-based
case we denote the tensor over $\Top$ by $X\otimes M$. The natural isomorphism in the non-based case reads
$$(X\times Y)\otimes M\cong X\otimes (Y\otimes M).$$
Forgetting base points turns a $\top$-enriched category $\scM$ into a $\Top$-enriched one. If $\scM$ is tensored
over $\top$ it is also tensored over $\Top$: we define
$$X\otimes M= X_+\boxtimes M$$
where $X_+=X\sqcup \{\ast\}$ with the additional point as base point.

\begin{exam}\label{3_1}
 $\Mon$ is $\top$-enriched, tensored and cotensored \cite[Prop. 2.10]{PRV}. The cotensor
$M^X$ is the $k$-function space with pointwise multiplication, $X\boxtimes M$ is more complicated: as a set, it is a 
free product of copies $M$, one copy for each $x\in X$ different from the base point. By the same argument as in \cite{PRV}
the category $\Sgp$ is $\Top$-enriched and tensored and cotensored over $\Top$. 

If $\otimes_{\Sgp}$ denotes the tensor in $\Sgp$ and $\otimes$ the one over $\Top$ in $\Mon$, then the universal properties of the tensor 
and of the adjunction of \ref{2_3b} imply that there is a natural isomorphism
$$(K\otimes_{\Sgp} G)_+\cong K\otimes (G_+)$$
in $\Mon$ for semigroups $G$.
\end{exam}

\begin{defi}\label{3_2a}
 Let $\scM$ be a $\Top$-enriched category. Two morphisms $f,g:A\to X$ are called homotopic if there is a path in
 $\scM(A,X)$ joining $f$ and $g$.
\end{defi}

Clearly, the homotopy relation is an equivalence relation preserved under composition. Passing to path components we obtain
the \textit{homotopy category} $\pi\scM$.

If $\scM$ is tensored over $\Top$ it has a canonical cylinder functor $M\mapsto I\otimes M$. The associated homotopy notion
coincides with the one of Definition \ref{3_2a}.

\begin{defi}\label{3_2}
Let $\mathcal{M}$ be a category and $\mathscr{W}$ a class of morphisms in $\mathcal{M}$, which we will call
\textit{weak equivalences}. The \textit{localization} of 
$\mathcal{M}$ with respect to $\mathscr{W}$ is a category $\mathcal{M}[\mathscr{W}^{-1}]$ with $\ob\mathcal{M}[\mathscr{W}^{-1}]=\ob \mathcal{M}$
and a functor
$\gamma: \mathcal{M}\to\mathcal{M}[\mathscr{W}^{-1}]$ such that \\
(1) $\gamma$ is the identity on objects\\
(2) $\gamma(f)$ is an isomorphism for all $f\in\mathscr{W}$\\
(3) if $F:\mathcal{M}\to\mathcal{D}$ is a functor such that $F(f)$ is an isomorphism for all $f\in \mathscr{W}$ then there exists a unique 
functor $\overline{F}:\mathcal{M}[\mathscr{W}^{-1}]\to\mathcal{D}$ such that $F=\overline{F}\circ \gamma $.
\end{defi}

\begin{prop}\label{3_3}
 Let $\mathcal{M}$ be a $\Top$-enriched tensored category and $\mathscr{W}$ a class of morphisms in $\mathcal{M}$ such that\\
(1) $\mathscr{W}$ contains all homotopy equivalences,\\
(2) there is a functor $Q:\scM\to \scM$ and a natural transformation $\varepsilon:Q\to \Id$ or a natural transformation
$\eta: \Id \to Q$ taking values in $\mathscr{W}$
such that $Qf$ is a homotopy equivalence for each $f\in \mathscr{W}$.\\
Then $\mathcal{M}[\mathscr{W}^{-1}]$ exists. 
Precisely, let $\mathcal{HM}$ be the category with  $\ob \mathcal{HC}=\ob \scM$ and $\mathcal{HM}(M_1,M_2)=\scM(QM_1,QM_2)$. Then
$\mathcal{M}[\mathscr{W}^{-1}]= \pi\mathcal{HC}$,
the quotient category obtained by passing to homotopy classes. The functor $\gamma :\scM\to \mathcal{M}[\mathscr{W}^{-1}]$ is 
the identity on objects and maps
a morphism $f$ to the homotopy class of $Qf$.
\end{prop}

\begin{proof}
The proof is essentially the same as in the case of a Quillen model category (e.g. see \cite[Thm 8.3.5]{Hirsch}). We recall
the construction of the localization $\mathcal{M}[\mathscr{W}^{-1}]$ in this case. So let $\mathcal{M}$ be a Quillen model
category, let $\varepsilon : C\to \Id$ respectively $\eta; \Id\to R$ be a cofibrant respectively fibrant replacement functor.
There are cylinder objects giving rise to the left homotopy relation.\\
\textit{Step 1:} Using the fact that $RC(X)$ is fibrant and cofibrant for each object $X$ in $\scM$ one proves that left homotopy
is an equivalence relation on $\scM(RC(A),RC(X))$ which is preserved under composition. Let $\pi\scM(RC(A),RC(X))$ be the set of
equivalence classes. One defines
$$ \ob \mathcal{M}[\mathscr{W}^{-1}]= \ob \scM\quad \textrm{and}\quad \mathcal{M}[\mathscr{W}^{-1}](A,B)= \pi\scM(RC(A),RC(B)),$$
and it follows that $\mathcal{M}[\mathscr{W}^{-1}]$ is a category.\\
\textit{Step 2:} One proves that $RC(f)$ is a homotopy equivalence if $f:A\to X$ is a weak equivalence. Then one defines 
$$\gamma :\scM\to \mathcal{M}[\mathscr{W}^{-1}] \qquad f\mapsto RC(f).$$
In particular, $\gamma$ maps weak equivalences to isomorphisms.\\
\textit{Step 3:}
One shows that a functor $F:\scM\to \scN$, which maps weak equivalences to isomorphisms, maps homotopic morphisms to the same morphism.\\
\textit{Step 4:} Given a functor $F:\scM\to \scN$, which maps weak equivalences to isomorphisms, then there is a unique functor
$\bar{F}: \mathcal{M}[\mathscr{W}^{-1}]\to \scN$ such that $F=\bar{F}\circ \gamma$, and $\bar{F}$ is defined on objects by
$\bar{F} (X)=F (X)$ and on morphisms $[f]\in \mathcal{M}[\mathscr{W}^{-1}](A,X)$ by
$$\bar{F}([f]) = F (\varepsilon(X))\circ (F (\eta(CX)))^{-1}\circ F (f)\circ F (\eta(CA))\circ (F(\varepsilon(A)))^{-1},$$
 where $[f]$ is the homotopy class of $f$.

We now prove Proposition \ref{3_3}. We deal with the case where we have
 a natural transformation $\varepsilon:Q\to \Id$ taking values in $\mathscr{W}$.\\
Step 1 follows from the topological enrichment
$$\ob \mathcal{M}[\mathscr{W}^{-1}]= \ob \scM\quad \textrm{and}\quad \mathcal{M}[\mathscr{W}^{-1}](A,B)= \pi\scM(Q(A),Q(B))$$
which is a category.\\
Step 2 holds by Assumption \ref{3_3}.2, and we define 
$$\gamma:\scM\to \mathcal{M}[\mathscr{W}^{-1}]\qquad f\mapsto Q(f).$$
 $\gamma$ maps weak equivalences to isomorphisms.\\
For Step 3 we need the cylinder functor: the bottom and top inclusions $i_0\otimes \id,i_1\otimes \id:X\cong \ast\otimes X\to I\otimes X$ into the cylinder 
are homotopy equivalences with the common homotopy inverse $r\otimes \id :I\otimes X\to \ast\otimes X\cong X$. \\
Step 4: Given a functor $F:\scM\to \scN$, which maps weak equivalences to isomorphisms, we define $\bar{F}: \mathcal{M}[\mathscr{W}^{-1}]\to \scN$ 
by
$$\bar{F} (X)=F (X)\quad\textrm{and}\quad \bar{F}([f]) = F (\varepsilon(X))\circ F(f)\circ (F(\varepsilon(A)))^{-1}$$
for $[f]\in \mathcal{M}[\mathscr{W}^{-1}](A,X)$. The rest follows like in  \cite[Thm 8.3.5]{Hirsch}.
\end{proof}

\begin{rema}\label{3_3c}
For Proposition \ref{3_3} we do not need that the tensor $X\otimes M$ exists for all topological spaces: it suffices that $\scM$ is tensored 
over the full subcategory of $\Top$ consisting of a point $\ast$ and the unit interval $I$.
\end{rema}

\begin{nota}\label{3_3a}
Following the standard convention we denote $\mathcal{M}[\mathscr{W}^{-1}]$ by $\Ho \scM$ if the class $\mathscr{W}$ has 
been specified.\\
A pair $(Q,\ {\varepsilon:Q\to \Id})$ respectively $(Q,\ {\eta:\Id \to Q})$ satisfying the requirements of \ref{3_3} will be called 
a \textit{cofibrant} respectively \textit{fibrant replacement functor}.
Each $\Top$-enriched category $\scM$ considered in this paper will have a continuous cofibrant replacement functor, and we call the category 
$\scH\scM$ the \textit{category of $Q$-morphisms} associated with $\scM$.
\end{nota}

\begin{defi}\label{3_3d}
A functor $Q:\scM\to \scM$ together with a natural transformation $\varepsilon: Q\to \Id$ is called a
\textit{strong cofibrant replacement functor} if each $\varepsilon(M) : Q(M)\to M$ is a weak equivalence
and $p_\ast: \scM(QA,B)\to \scM(QA,C)$ is a homotopy equivalence whenever $p:B\to C$
is a weak equivalence.
\end{defi}
Clearly, a strong cofibrant replacement functor is a cofibrant replacement functor.

\begin{leer}\label{3_4} \textbf{Examples:}
\begin{enumerate}
\item Let $\mathscr{W}\subset\Mon$ be the class of weak equivalences in the sense of \ref{2_3c}. 
Then $WV:\Mon\to \Mon$ together with $WVM\xrightarrow{\varepsilon(VM)}VM
\xrightarrow{q(M)} M$ is a strong cofibrant replacement functor, and the
$Q$-morphisms are the $huh$-morphisms. This follows from informations in 
\ref{2_2}, \ref{2_4}, \ref{2_6b}, and \ref{2_7}.
\item  Let $\mathscr{W}\subset\Monw$ be again the class of weak equivalences. Then $W:\Monw\to \Monw$ together with $\varepsilon: W\to \Id$
is a strong cofibrant replacement functor, and the
$Q$-morphisms are the $uh$-morphisms. The required information is obtained from \ref{2_2}, \ref{2_4}, and \ref{2_7}.
\item Let $\mathscr{W}\subset\Sgp$ be the class of weak equivalences. Then $\overline{W}:\Sgp\to \Sgp$ together with $\overline{\varepsilon}:
\overline{W}\to \Id$ is a strong cofibrant replacement functor, and the
$Q$-morphisms are the $h$-morphisms by informations from \ref{2_2} and \ref{2_4}.
\item Let $\mathscr{W}\subset\top$ be the class of based maps which are (not necessarily based) homotopy equivalences. 
Then $V^t:\top\to\top$ together with 
$q:V^t\to \Id$ is a strong cofibrant replacement functor by the lemma below, the proof of which we leave as an exercise.
\item Let $\mathscr{W}\subset\topw$ be the class of homotopy equivalences. Then $\Id: \topw\to \topw$ is a strong
 cofibrant replacement functor and each map is
a $Q$-morphism.
\end{enumerate}
\end{leer}
\begin{lem}\label{3_3e}
 Let $A$ be a well-pointed space and $p:X\to Y$ a map in $\top$ which is a not necessarily based homotopy equivalence. Then
$$
p_\ast: \top(A,X)\to \top(A,Y)
$$
is a homotopy equivalence in $\Top$.
\end{lem}

\begin{prop}\label{3_3f}
The localizations of the categories of \ref{3_4} with respect to their weak equivalences exist.
\end{prop}
\begin{proof}
We apply \ref{3_3} and \ref{3_3c}. We have to show that our categories are tensored over  
the full subcategory of $\Top$ consisting of a point $\ast$ and the unit interval $I$, the other
assumptions of \ref{3_3} have been verified above. \\
We already know that $\Mon$ and $\Sgp$ are tensored over $\Top$. The category $\top$ is tensored over itself
by the smash product and hence also tensored over $\Top$.  
For the Examples \ref{3_4}.2 and \ref{3_4}.5 it suffices to know
that for any object $M$ in the category the tensor $I\otimes M$ is well-pointed (recall $\ast\otimes M\cong M$).
This is well known
for $\topw$ and holds for $\Monw$ by \cite[Prop. 7.8]{PRV}.
\end{proof}

\begin{defi}\label{3_5}
Let $\mathcal{M}$ be a category and $\mathscr{W}$ a class of morphisms in $\mathcal{M}$ such that $\mathcal{M}[\mathscr{W}^{-1}]$ exists.
Let $F:\scM\to \scD$ be a functor. A functor $\bL F:\mathcal{M}[\mathscr{W}^{-1}]\to \scD$ together with a natural
 transformation $\tau: \bL F\circ \gamma \to F$
is called  \textit{left derived functor} of $F$, if given any functor $T:\mathcal{M}[\mathscr{W}^{-1}]\to \scD$ 
and natural transformation
$\sigma :T\circ \gamma \to F$, there is a unique natural transformation $\rho: T\to \bL F$ such that $\sigma=\tau\circ (\rho\ast\gamma)$.\\
Dually, a functor $\bR F:\mathcal{M}[\mathscr{W}^{-1}]\to \scD$ together with a natural transformation $\mu: F\to \bR F\circ \gamma$
is called  \textit{right derived functor} of $F$, if given any functor $G:\mathcal{M}[\mathscr{W}^{-1}]\to \scD$ and natural transformation
$\nu :F\to G\circ \gamma$, there is a unique natural transformation $\xi : \bR F\to G$ such that $(\xi\ast\gamma)\circ\mu$.
\end{defi}

\begin{rema}\label{3_5a}
(1) A left or right derived functor is unique up to natural isomorphism if it exists.\\
(2) If $F:\scM\to \scD$ maps weak equivalences to isomorphisms, then the induced functor 
$\overline{F}: \mathcal{M}[\mathscr{W}^{-1}]\to \scD$ is the right and
left derived functor of $F$.
\end{rema}

\begin{prop}\label{3_6a}
Let $\scM$ be as in Proposition \ref{3_3}, and let 
$F:\scM\to \scB$ be a functor which maps homotopy equivalences to isomorphisms. Then $\bL F :\mathcal{M}[\mathscr{W}^{-1}]\to \scB$
exists if $\scM$ has a cofibrant replacement functor, and $\bR F :\mathcal{M}[\mathscr{W}^{-1}]\to \scB$
exists if $\scM$ has a fibrant replacement functor. In both cases the derived functor is induced by
$F\circ Q:\scM\to \scB$.
\end{prop}

\begin{proof}
The proof is the same as in the case of a model category (e.g. see \cite[8.4.]{Hirsch}).  
\end{proof}

 Let $F:\scM\to \scB$ be a functor between $\Top$-enriched categories admitting cofibrant replacement functors
$Q_{\scM}:\scM\to \scM$ and $Q_{\scB}:\scB\to \scB$. Proposition \ref{3_6a} motivates the introduction of the
functor
\begin{leer}\label{3_6b}\hspace*{27ex}{$F^{\scH}: \scH\scM\to \scH\scB$}
\end{leer}
\hspace*{-3ex}defined on objects by $F^{\scH}(X)=F(Q_{\scM}X)$ and on morphisms by
$$
F^{\scH}: \scM(Q_{\scM}X,Q_{\scM}Y)\xrightarrow{Q_{\scB}\circ F} \scB(Q_{\scB}FQ_{\scM}X,Q_{\scB}FQ_{\scM}Y).
$$
If $F$ preserves homotopy equivalences, e.g. if $F$ is continuous, and $\pi_{\scB}:\scB\to \pi\scB$ is the canonical
functor, then $\pi_{\scB}\circ F^{\scH}$ induces the left derived functor 
$$
\Ho F:\Ho\scM\to \Ho\scB$$
of $\pi_{\scB}\circ F$. Following model category terminology, we call $\Ho F$ the \textit{total left derived functor
of} $F$.

One of the objectives of this paper is to show that the classifying space functor and the Moore loop space functor induce
an adjoint derived pair (see Theorem \ref{4_3b} below). This is the path-component version of the more general result (Theorem \ref{4_3a} below)
that 
$$
\xymatrix{
\BH:\mathcal{H}\Mon \ar@<0.5ex>[r] & \top:\OH \ar@<0.5ex>[l]
}
$$
are a homotopically adjoint pair. To make this last statement precise we need some preparations.

\begin{defi}\label{3_6}
Let $\mathcal{A}$ and $\mathcal{B}$ be topologically enriched categories. A functor $F:\mathcal{A}\to\mathcal{B}$ is called \textit{continuous} if
$$
F:\mathcal{A}(A,B)\longrightarrow\mathcal{B}(FA,FB)
$$
is continuous for all $A$ and $B$ in $\mathcal{A}$. 

If $F,G:\mathcal{A}\to\mathcal{B}$ are continuous functors, a collection of morphisms $\{\alpha(A): FA\to GA;\; A\in\ob\mathcal{A}\}$ 
is called a \textit{natural transformation up to homotopy} if the diagram
$$
\xymatrix{
\mathcal{A}(A,B) \ar[rr]^F \ar[d]^G
&& \mathcal{B}(FA,FB) \ar[d]^{\alpha(B)_\ast}
\\
\mathcal{B}(GA,GB)\ar[rr]^{\alpha(A)^\ast} && \mathcal{B}(FA,GB)
}
$$
 is homotopy commutative.
 
A pair of continuous functors
$$F:\scA \leftrightarrows \scB:G$$
is called a \textit{homotopy adjoint pair} if there is a natural transformation up to homotopy
$$\alpha(A,X):\scB(FA,X)\to \scA(A,GX)$$
such that each $\alpha(A,X)$ is a homotopy equivalence. The homotopy equivalences are called the
homotopy adjunctions.
\end{defi}

Just as the usual notion of adjunction is equivalently encoded by the concepts of unit and counit,
Proposition \ref{3_9} below describes how a homotopy adjunction is specified by a homotopy unit and 
a homotopy counit.

Observe that we have chosen a strong form of a natural transformation $\alpha : F\to G$ up to homotopy: for
each morphism $f:A\to B$ in $\scA$ we have a square
$$\xymatrix{
FA\ar[rr]^{\alpha(A)} \ar[d]_{Ff} && GA\ar[d]^{Gf}\\
FB\ar[rr]^{\alpha(B)} && GB
}$$
commuting up to a homotopy $H(f)$ which is continuous in $f$.

The proofs of the following two lemmas are easy exercises.

\begin{lem}\label{3_7} Let $S,T,U:\scA\to \scB$ be continuous functors of topologically enriched categories.

(1) Each natural transformation $\alpha:S\to T$ is a natural transformation up to homotopy.

(2) If $\varepsilon :S\to T$ and $\eta:T\to U$ are natural transformations up to homotopy, then $\eta\circ \varepsilon:S\to U$ is one.

(3) Let $\varepsilon :S\to T$ be a natural transformation up to homotopy such that each $\varepsilon(A)$ is a homotopy equivalence.
 Choose a homotopy inverse $\eta(A)$ of 
$\varepsilon(A)$ for each $A$ in $\mathcal{A}$. 
Then the $\eta(A)$ form a natural transformation $\eta :T\to S$ up to homotopy.\hfill\ensuremath{\Box}
\end{lem}

\begin{lem}\label{3_8} Let $S,T,U,V:\scA\to \scB$ be continuous functors of topologically enriched categories, and let $\varepsilon:S\to T$ 
and $\eta:U\to V$ be natural transformations up to homotopy.

(1) Let $F,G:\scA^{\op}\times \scA\to \Top$ be defined by $F(A,B)=\scA(A,B)$ and $G(A,B)=\scB(TA,TB)$. Then 
$$\tau(A,B):\scA(A,B)\xrightarrow{T} \scB(TA,TB)$$
is a natural transformation from $F$ to $G$.

(2) Let $F,G:\scA^{\op}\times \scA\to \Top$ be defined by $F(A,B)=\scB(VA,SB)$ and $G(A,B)= \scB(UA,TB)$. Then
$$\alpha(A,B):\scB(VA,SB)\xrightarrow{\varepsilon(B)_\ast\circ\eta(A)^\ast} \scB(UA,TB)$$
is a natural transformation from $F$ to $G$ up to homotopy.\hfill\ensuremath{\Box}
\end{lem}

\begin{prop}\label{3_9} Let $F:\scA \leftrightarrows \scB:G$ be a pair of continuous functors of topologically enriched categories.
 Suppose there are natural transformations up homotopy
$$\mu(A): A\to GF(A)\qquad\textrm{and}\qquad \eta(X):FG(X)\to X$$
such that $$G(\eta(X))\circ \mu(GX)\simeq \id_{GX}\qquad\textrm{and}\qquad \eta(FA)\circ (F(\mu(A))\simeq \id_{FA}.$$
Then $F$ and $G$ are a homotopy adjoint pair. (We call $\mu: \Id \to GF$ the \textit{homotopy unit} and $\eta:FG\to \Id$
the \textit{homotopy counit} of the resulting homotopy adjunction.)
\end{prop}

\begin{proof}
 We define 
$$
\alpha(A,X):\scB(FA,X)\xrightarrow{G} \scA(GFA,GX)\xrightarrow{\mu(A)^\ast} \scA(A,GX)$$
and
$$
\beta(A,X):\scA(A,GX)\xrightarrow{F}\scB(FA,FGX)\xrightarrow{\eta(X)_\ast}\scB(FA,X).$$
By \ref{3_8} both are natural transformations up to homotopy. The following diagram
shows that $\beta(A,X)\circ \alpha(A,X)\simeq \id$.
$$\xymatrix
{
&& \scA(GFA,GX)\ar[dd]^F\ar[rr]^{\mu(A)^\ast}   && \scA(A,GX)\ar[dd]^F\\
&& & \textrm{II} & \\
\scB(FA,X)\ar[rruu]^G\ar[ddrr]^{\eta(FA)^\ast} & \textrm{I} & \scB(FGFA,FGX)\ar[rr]^{(F\mu(A))^\ast}\ar[dd]^{\eta(X)_\ast} &&
 \scB(FA,FGX)\ar[dd]^{\eta(X)_\ast}\\
&& & \textrm{III} & \\
&& \scB(FGFA,X)\ar[rr]^{(F\mu(A))^\ast} && \scB(FA,X)
}
$$
The squares II and III commute and square I commutes up to homotopy, and $(F\mu(A))^\ast\circ \eta(FA)^\ast\simeq \id$ by
assumption.\\
 The proof that $\alpha(A,X)\circ \beta(A,X)\simeq \id$ is dual.
\end{proof}
\begin{defi}\label{3_18}
A homotopy adjunction
$F: \scA \leftrightarrows \scB :G$
is called \textit{natural} if there is a natural homotopy equivalence
$$\beta(A,X):\scA(A,GX)\to \scB(FA,X)$$
and  \textit{conatural} if there is a natural homotopy equivalence
$$\alpha(A,X):\scB(FA,X)\to \scA(A,GX)$$
(because in this case there is a natural homotopy unit, respectively, a natural
homotopy counit).
\end{defi}

\section{The classifying space and the Moore loop space functor}
\begin{leer}\label{4_1}
 \textbf{The 2-sided bar construction:} Let $\scC$ be a small topologically enriched category, $X$ a $\scC^{op}$-diagram
and $Y$ a $\scC$-diagram in $\Top$. 
 We define a simplicial space $B_\bullet (X,\scC ,Y)$ by
$$
\begin{array}{rcl}
B_0(X,\scC,Y) & = & \coprod_{A \in \scC}X(A)\times Y(A) \\
B_n(X,\scC,Y)& =& \coprod_{A,B\in \scC} X(B)\times \scC_n(A,B)\times Y(A)\quad \textrm{for } n>0,
\end{array}
$$
where $\scC_n(A,B)$ is the space of all composable $n$-tuples of morphisms $(f_1,\ldots,f_n)$ such that source$(f_n)\ =\ A$ and target$(f_1)\ =\ B$,
with boundary and degeneracy maps given by
$$
\begin{array}{ll}
d^i(x,f_1,\ldots,f_n,y)= (X(f_1)(x),f_2,\cdots,f_n,y) & i=0\\
d^i(x,f_1,\ldots,f_n,y)= (x,f_1,\ldots,f_i\circ f_{i+1},\ldots,f_n,y) & 0<i<n\\
d^i(x,f_1,\ldots,f_n,y)= (x,f_1,\ldots,f_{n-1},Y(f_n)(y))& i=n\\
s^i(x,f_1,\ldots,f_n,y)= (x,f_1,\ldots,f_i,\id, f_{i+1},\ldots,f_n,y) & 0\le i\le n
\end{array}
$$
Let $B(X,\scC,Y)=|B_\bullet (X,\scC ,Y)|$ be its topological realization.
\end{leer}

We consider a topological monoid as a topologically enriched category with one object and define the \textit{classifying space functor}
$$
B:\Mon\longrightarrow\top
$$
by $BM=B(\ast,M, \ast)$. 
Since $BM$ is well-pointed if $M$ is, the classifying space functor is a functor of pairs
$$B:(\Mon,\Monw)\to (\top,\topw).$$

\begin{leer}\label{4_1a}
 We will also work with the variant
$$
\widetilde{B}:\Mon\longrightarrow\top
$$
where the topological realization of $ B_\bullet (\ast,M, \ast)$ is replaced by the fat realization which disregards degeneracies.
Since the fat realization does not make use of identities the functor $\widetilde{B}$ extends to $\Sgp$; moreover,
$\widetilde{B}G$ is well-pointed for any semigroup $G$ so that 
$$\widetilde{B}:\Sgp \to \topw.$$
By construction, there is a natural homeomorphism $\widetilde{B}(G)\cong B(G_+)$ for semigroups $G$, and the diagram
$$\xymatrix{
\widetilde{B}(M)\ar[rr]^\cong\ar[dr]_{p(M)} & & B(M_+)\ar[ld]^{B(\kappa(M))}\\
&B(M)&
}
$$
commutes for monoids $M$, 
where $\kappa: M_+\to M$ is the counit of the adjunction
\ref{2_3b} and $p: \widetilde{B}\to B$ is the natural projection.

It is well-known that $p(M):\widetilde{B}(M)\to B(M)$ and hence $B(\kappa(M)):B(M_+)\to B(M)$ are homotopy equivalences if
$M$ is well-pointed.
\end{leer}

\begin{leer}\label{4_2}
\textbf{The Moore path and loop space}: Let $X$ be a (not necessarily based) space. 
The \textit{Moore path space} of $X$ is the subspace $\P (X)\subset X^{\mathbb{R}_+}\times\mathbb{R}_+$ consisting of 
all pairs $(w,r)$ such that $w(t)=w(r)$ for all $t\ge r$. 
We call $r$ the \textit{length} of $w$ and denote it by $r=l(w)$.\\
For two paths $(w_1,r_1)$ and $(w_2,r_2)$ with $(w_1)(r_1)= (w_2)(0)$ we define
\textit{path addition} by
$$
(w_1,r_1)+(w_2,r_2)=(w, r_1+r_2)
$$
with 
$$
w(t)=\left\{
\begin{array}{ll}
w_1(t), & 0\le t\le r_1,\\
w_2(t-r_1) & r_1 \le t.
\end{array}\right.
$$

If $(X,\ast)$ is a based space, the \textit{Moore loop space} $\Omega'(X)\subset \P (X)$ is the subspace of 
all pairs $(w,r)$ with $(w)(r)=(w)(0)=\ast$. Path addition defines a monoid structure on $\Omega'X$ with
$(c,0)$ as unit, where $c:\mathbb{R}_+\to X$ is  the constant map to $\ast$. The usual loop space $\Omega X$ is 
embedded in $\Omega'(X)$ as a deformation retract.

It follows from \cite[(11.3)]{DKP} that $\Omega'(X)$ is well-pointed if $X$ is. Hence $\Omega'$ defines a functor of pairs
$$\Omega': (\top,\topw)\to(\Mon,\Monw).$$
\end{leer}

Following \ref{3_6b} we  have pairs of continuous functors
$$\BH : \scH\Mon\leftrightarrows \scH\top:\OH$$
and
$$\BwH:\scH\Monw\leftrightarrows \scH\topw =\topw:\OwH.$$
We shall prove

\begin{theo}\label{4_3}
The functors
$$
\xymatrix{
\BwH:\mathcal{H}\Monw \ar@<0.5ex>[r] & \topw:\OwH \ar@<0.5ex>[l]
}
$$
are a conatural homotopically adjoint pair: There is a continuous natural map 
$$
\lambda(WM,X): \Mon(WM, W\Omega'X)\longrightarrow\top(BWM,X)
$$
which is a homotopy equivalence.
\end{theo}

As an immediate consequence we obtain

\begin{theo}\label{4_3a}
The functors
$$
\xymatrix{
\BH:\mathcal{H}\Mon \ar@<0.5ex>[r] & \scH\top:\OH \ar@<0.5ex>[l]
}
$$
are a conatural homotopically adjoint pair: There is a continuous natural map 
$$
\lambda(WVM,V^tX): \Mon(WVM, WV\Omega'V^tX)\longrightarrow\top(V^tBWVM,V^tX)
$$
which is a homotopy equivalence. 
\end{theo}
\begin{proof}
 Replacing $M$ by $VM$ and $X$ by $V^tX$ in Proposition \ref{4_3} we obtain a natural
homotopy equivalence
$$\Mon(WVM,WV\Omega'V^tX)\xrightarrow{\simeq} \top(BWVM,V^tX).$$
Since $BWVM$ is well-pointed the natural map $q(BWVM):V^tBWVM \to BWVM$ is a based
homotopy equivalence inducing a natural homotopy equivalence
$$ q(BWVM)^\ast : \top(BWVM,V^tX) \to \top(V^tBWVM,V^tX).$$
\end{proof}

Passing to homotopy classes (see \ref{3_3}) we obtain 

\begin{theo}\label{4_3b}
 The functors
$$\Ho B:\Ho \Mon\leftrightarrows \Ho \top:\Ho \Omega'$$
are an adjoint pair. Moreover, $\Ho B$ is the left derived of $\gamma_{\top}\circ B$ and $\Ho \Omega'$
the left derived of $\gamma_{\Mon}\circ \Omega'$.
\end{theo}
\begin{proof}
 This follows from our explicit description of the localizations and the derived functors in Section 3.
\end{proof}

The rest of this Section is devoted to the proof of
 Theorem \ref{4_3}. By \ref{3_9} it suffices to construct a homotopy unit $\mu: \Id_{\scH\Monw}\to \OwH\BwH $ and a homotopy counit
$\eta: \BwH\OwH \to \Id_{\scH\topw}$. Then $\lambda(WM,X)$ is the composite
$$
\Mon(WM, W\Omega'X)\xrightarrow{B}\top(BWM,BW\Omega'X)\xrightarrow{\eta(X)_\ast} \top(BWM,X).$$

\begin{leer}\label{4_4} This means, we have to construct continuous homomorphisms
$$\mu(WM):WM\longrightarrow W\Omega'BWM $$
which constitute a natural transformation up to homotopy with respect to 
homomorphisms $WM\to WN$, and a natural transformation
$$\eta(X): BW\Omega'X\longrightarrow X,$$
such that\\
(1) $W\Omega'\eta(X)\circ\mu (W\Omega'X)\simeq\id_{W\Omega'X}$ in $\Monw$ and \\
(2) $\eta(BWM)\circ B\mu (WM)\simeq\id_{BWM}$ in $\topw$.\\
(For $\lambda$ to be a natural transformation we need $\eta$ to be a natural transformation.)
\end{leer}

\begin{leer}\label{4_5}
\textbf{The homotopy counit:} Let $X$ be a based space and let
$$
\Delta^n=\{(t_0,\ldots,t_n)\in\mathbb{R}^{n+1}; \quad \sum\limits^n_{i=0}t_i=1, \; t_i\ge 0\textrm{ for all } i\}
$$
denote the standard $n$-simplex. The \textit{evaluation map} 
$$
\ev(X): B\Omega' X=\left(\coprod\limits_{n\ge 0}(\Omega'X)^n\times\Delta^n\right)/\sim \ \longrightarrow X
$$
is defined by
$$
\ev(X)((w_1,\ldots,w_n)(t_0,\ldots,t_n)) =(w_1+\ldots+w_n)\left(\sum\limits^n_{i=1}t_i\cdot\sum\limits^i_{j=1}l(w_j)\right)
$$
where $l(w_j)$ is the length of $w_j$.
\end{leer}

The homotopy counit $\eta$ is the natural map
$$
\xymatrix{
\eta(X): BW\Omega'X \ar[rr]^(0.6){B\varepsilon(\Omega'X)}
&& B\Omega'X\ar[rr]^(0.5){\ev(X)} && X.
}
$$

\begin{leer}\label{4_6}
\textbf{The homotopy unit:} For a monoid $M$ let
$EM$ denote the 2-sided bar construction $B(M,M,\ast)$. Then
$$
z\cdot(x_0,x_1,\ldots,x_n)=(z\cdot x_0,x_1,\ldots,x_n)
$$
defines a left $M$-action on the simplicial space  $B_\bullet (M,M,\ast)$ and hence on $EM$.
\end{leer}

Let $P(EM)$ denote the space of Moore paths in $EM$ starting at the base-point $(e)$ in the 0-skeleton $M$ of $EM$. The endpoint projection
$$
P(EM)\longrightarrow EM
$$
is known to be a fibration. Moreover, it is a homotopy equivalence because $P(EM)$ and $EM$ are contractible. Let $P(EM,M)$ be the pullback
$$
\xymatrix{
P(EM,M) \ar[rr] \ar[d]^{\pi(M)}
&& P(EM) \ar[d]
\\
M \ar[rr]^i && EM
}
$$
where $i$ is the inclusion of the 0-skeleton, i.e. $P(EM,M)$ is the space of Moore paths in $EM$ starting at $(e)$ and ending in $M$. 
Then $\pi(M)$ is a fibration and a homotopy equivalence.
We define a monoid structure $\oplus$ in $P(EM,M)$ by
$$
w_1\oplus w_2=w_1+x\cdot w_2
$$
where $+$ is the usual path addition, $x\in M$ is the endpoint of $w_1$, and $x\cdot w_2$ is the path $t\mapsto x\cdot w_2(t)$. 
Then $\pi(M):P(EM,M)\to M$ is a homomorphism and hence a weak equivalence of monoids.

Factoring out the operation of $M$ on $EM$ we obtain a projection
$$
EM\to BM
$$
inducing a homomorphism
$$
\rho'(M):(P(EM,M),\oplus)\longrightarrow (\Omega'BM,+).
$$
Since we do not know whether or not $(P(EM,M))$ is well-pointed we apply the whiskering process to it and obtain a homomorphism
$$
\rho(M): V(P(EM,M),\oplus)\xrightarrow{q((P(EM,M))} (P(EM,M),\oplus)\xrightarrow{ \rho'(M)} (\Omega'BM,+).
$$
The homomorphism $\sigma (M): WV(P(EM,M)\to M$ defined by
$$\xymatrix{
WVP(EM,M)\ar[rrr]^{\varepsilon(VP(EM,M))}\ar[rrrd]_{\sigma (M)} &&& VP(EM,M)\ar[rr]^{q(P(EM,M))} &&
 P(EM,M)\ar[lld]^{\pi (M)} \\
&&& M &&
}
$$
is a weak equivalence.
All these constructions are functorial in $M$ and the maps between them are natural in $M$. 
We apply them to $WM$ rather than to $M$; in particular  $ \sigma(WM)$ is a homotopy equivalence in $\Monw$.

We choose a homotopy inverse of $ \sigma(WM)$ in $\Monw$
$$
\nu(WM):WM\longrightarrow WVP(EWM,WM),
$$
which is a natural transformation up to homotopy with respect to homomorphisms $WM\to WN$ by Lemma \ref{3_7}.

We define our homotopy unit by
$$
\xymatrix@C=4ex{
\mu(WM): WM \ar[rr]^(.45){\nu(WM)} 
&& WVP(EWM,WM) \ar[rr]^(.56){W\rho(WM)} 
&&
W\Omega'BWM,
}
$$
which is a natural transformation up to homotopy by Lemma \ref{3_7}.

Our verification of the conditions \ref{4_4} depends on an explicit description of an $h$-morphism $M\to\Omega'BM$ defined by a 
natural homomorphism
$$
\zeta '(M) :\overline{W}(M) \longrightarrow\Omega'BM
$$
and the interplay of $\overline{W}(M)$ and $WM$. 

We define $\zeta '(M)$ as a composite of homomorphisms
$$
\xymatrix{
\overline{W}(M)\ar[rr]^{\zeta(M)} && P(EM,M)\ar[rr]^(.6){\rho' (M)} && \Omega'BM
}
$$
The homomorphism $\zeta(M)$ maps the element represented by $(x_0,t_1,\ldots,x_n)$ to the path 
$$
v_0+v_1+\ldots+v_n
$$
of length $t_1+\ldots+t_n+1$ in the simplex $(e,x_0,x_1,\ldots,x_n)\times\Delta^{n+1}\subset EM$, where
$$
v_k(s)=(e,x_0,\ldots,x_n)\times(u_0,\ldots,u_{n+1})\quad\textrm{and}\quad l(v_k)=t_{k+1}
$$
with
$$
u_r=\left\{
\begin{array}{ll}
(1-s)\cdot t_r\cdot\prod^k_{j=r+1}(1-t_j) & r\le k\\
s & r=k+1\\
0 & r\ge k+2
\end{array}
\right.
$$
and the conventions that $t_0=1$ and $t_{n+1}=1$.

Observe that $+$ is the usual path addition of Moore paths in $EM$ and not the monoid structure of $P(EM,M)$.

\pagebreak
\textbf{Example:}  $(x_0,t_1,x_1,t_2,x_2)$ is mapped to the path $v_0+v_1+v_2$ of length $t_1+t_2+1$ given by

\begin{figure}[h]\centering
  \begin{tikzpicture}[scale=1.2]
   \draw[] (0,0) -- (6,0) -- (3,5) -- (0,0);
   \draw[] (0,0) -- (3,2);      
   \draw[] (3,2) -- (6,0);   
   \draw[] (3,2) -- (3,5);
   \draw[>=stealth',->,line width=1pt] (1,0.67) -- node [below] {\scriptsize $v_1$} (2,0.533); 
   \draw[>=stealth',->,line width=1pt] (0,0) -- node [below] {\scriptsize $v_0$}  (1,0.67);   
   \draw[,line width=1pt] (2,0.533) -- node [left] {\scriptsize $v_2$} (3,5); 
   \draw[>=stealth',->,line width=1pt] (2,0.53)  (2.5,2.73);    
   \draw[dashed] (2,0.533) -- (6,0);        
   \draw[fill] (0,0) circle(2pt) node[below=1mm] {$e$};
   \draw[fill] (6,0) circle(2pt) node[below=1mm] {$x_0x_1$};
   \draw[fill] (3,5) circle(2pt) node[above=1mm] {$x_0x_1x_2$};
   \draw[fill] (3,2) circle(2pt) node[right=1mm] {$x_0$}; 
   \draw[fill] (1,0.67) circle(2pt) node [above] {\scriptsize $t_1$} ;
   \draw[fill] (2,0.533) circle(2pt) node [below] {\scriptsize $t_2$};         
  \end{tikzpicture}
\caption{}
\end{figure}

\begin{leer}\label{4_7}
By construction, $\pi(M)\circ\zeta(M)=\overline{\varepsilon}(M)$. In particular, $\zeta(M):\overline{W}M\to P(EM,M)$ is a weak equivalence of semigroups.
\end{leer}

\begin{rema}\label{4_7b}
 We will show below that $\rho'(M):P(EM,M)\to \Omega'BM$ is a weak equivalence if $M$ is grouplike, so that
$\rho'(M)\circ \zeta(M)$ is an $h$-morphism which is a weak equivalence if $M$ is grouplike. It is well-known
that such an $h$-morphism exists, but to our knowledge there is no explicit description in the literature. 
\end{rema}

\begin{leer}\label{4_7a}
Consider the following diagram
$$
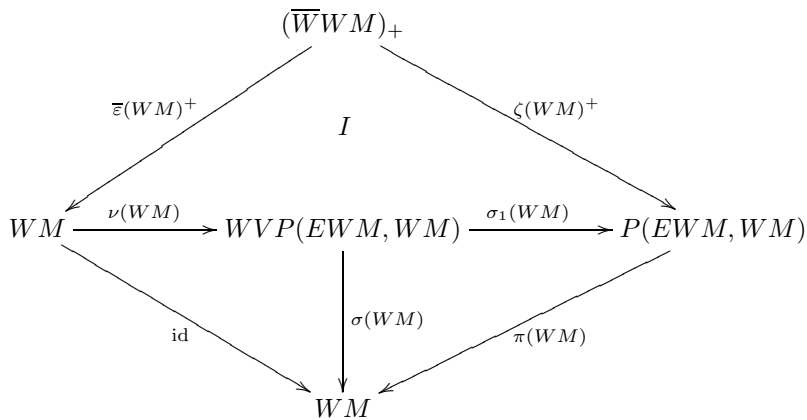
\xymatrix {
 && (\overline{W}WM)_+ \ar[lldd]_{\overline{\varepsilon}(WM)^+}\ar[rrdd]^{\zeta(WM)^+} &&\\
&& {\textrm{I}}&& \\
WM\ar[rr]^(0.35){\nu(WM)}\ar[rrdd]_{\id} && WVP(EWM,WM)\ar[rr]^{\sigma_1(WM)}\ar[dd]^{\sigma(WM)} && 
P(EWM,WM)\ar[lldd]^{\pi(WM)}\\
\\
&& WM &&
}
$$
where $\sigma_1(WM)=q(P(EM,M))\circ \varepsilon(VP(EM,M))$ and $f^+:G_+\to M$ is the adjoint of the homomorphism $f:G\to M$ from
a semigroup into a monoid.
By definition of $\nu(WM)$ and $\sigma(WM)$  the left lower triangle commutes up to homotopy in $\Monw$ and 
the right lower triangle is commutative.
Since
$$
\pi(WM)\circ \zeta(WM)=\overline{\varepsilon}(WM)\simeq \pi(WM)\circ\sigma_1(WM)\circ \nu(WM)\circ \overline{\varepsilon}(WM)
$$
Proposition \ref{2_4} implies that
$$ \sigma_1(WM)\circ \nu(WM)\circ \overline{\varepsilon}(WM)\simeq \zeta(WM) \qquad \textrm{in }\ \Sgp$$
which in turn is equivalent to the saying that
 square I commutes up to homotopy in $\Mon$. 
\end{leer}
We are now in the position to prove

\begin{prop}\label{4_8}
$\eta(BWM)\circ B\mu(WM)\simeq \id_{BWM}$ in  $\topw$.
\end{prop}

This result is a fairly easy consequence of

\begin{lem}\label{4_9} The diagram
$$
\xymatrix@C=4ex{
\widetilde{B}\overline{W}M \ar[rrrrrr]^{\widetilde{B}\overline{\varepsilon}(M)} \ar[d]_{\widetilde{B}\zeta(M)}
&&&&&&
 \widetilde{B}M \ar[d]^{p(M)}
\\
\widetilde{B}P(EM, M) \ar[rr]^(.55){\widetilde{B}\rho'(M)}
&& \widetilde{B}\Omega'BM \ar[rr]^{p(\Omega'BM)}
&& {B}\Omega'BM \ar[rr]^(.55){\ev(BM)}
&& BM
}
$$
commutes up to homotopy.
\end{lem}

\begin{proof}
 Let $f=\ev(BM)\circ p(\Omega'BM) \circ \widetilde{B}\rho'(M)\circ \widetilde{B}\zeta$ and let
$g= p(M)\circ \widetilde{B}\overline{\varepsilon}(M)$. Let
$z=(z_1,\ldots ,z_n)$ be an element in $(\overline{W}M )^n$, so that $z\times \Delta^n$ is an $n$-simplex
in $\widetilde{B}\overline{W}M$. If $z_j=(x_{j0},t_{j1},\ldots, x_{jr_j})$, then $f$ maps $z\times \Delta^n$
to the image of the path $\rho'(M)\circ \zeta(z_1)+\ldots +\rho'(M)\circ \zeta(z_n)$ which lies 
in the simplex
$$
\sigma = \sigma(z)=(x_{10},\ldots, x_{1r_1},\ldots,x_{n0},\ldots,x_{nr_n})\times \Delta^{r_1+\cdots +r_n+n}
$$
in $BM$, while $g$ maps $z\times \Delta^n$ identically (modulo possible degenerations) onto the simplex
$$
\tau =\tau(z)=(x_{10}\cdot \ldots \cdot x_{1r_1},\ldots , x_{n0}\cdot \ldots\cdot x_{nr_n})\times \Delta^n
$$
in $BM$, which is a face of $\sigma$. So $f|z\times \Delta^n$ is homotopic to $g|z\times \Delta^n$ by a 
linear homotopy. We call a homotopy from $f$ to $g$ \textit{admissible} if it maps $z\times \Delta^n$ to $\sigma(z)$
throughout the homotopy.

We are going to construct an admissible homotopy $H: \widetilde{B}\overline{W}M \times I \to BM$ from $f$ to $g$ by
induction on the canonical filtration $(\widetilde{B}\overline{W}M)^{(n)}$ of $\widetilde{B}\overline{W}M$.

$(\widetilde{B}\overline{W}M)^{(0)}$ is a point, which is mapped by $f$ and $g$ to the base-point. Now suppose that we 
have constructed an admissible homotopy
$$
H:(\widetilde{B}\overline{W}M)^{(n-1)}\times I \to BM.
$$
Let $z\times \Delta^n$ be an $n$-simplex in $\widetilde{B}\overline{W}M$ as above. We define 
$$q(z)=q(z_1,\ldots ,z_n)=r_1+\cdots +r_n \in \mathbb{N}$$
 and we extend $H$ over $(\widetilde{B}\overline{W}M)^{(n)}\times I$
by induction on $q$. 

If $q=0$, then $z=(z_1,\ldots ,z_n)$ with $z_j=(x_{j0})$ for $j=1,\ldots ,n$ and $\sigma(z)=\tau(z)=(x_{10},\ldots ,x_{n0})\times \Delta^n$.
Hence the space of all $n$-simplices $z\in (\overline{W}M)^n$ with $q(z)=0$ is $M^n$.
By induction, we have to find a homotopy
$$
h:M^n\times \Delta^n\times I\to M^n\times \Delta^n
$$
over $M^n$ which is already determined on $M^n\times \partial (\Delta^n\times I)$. If $b_n$ denotes the barycenter of $\Delta^n$ we map
$((x_1,\ldots,x_n),b_n,\frac{1}{2})$ to $((x_1,\ldots,x_n),b_n)$ and cone off.

If $q>0$ we have $q$ coordinates $t_{jk}\in I$ in $z$. So the space of all elements $z$ with $q(z)=q$ is the union of spaces of
the form $M^{n+q}\times I^q$ which may intersect on their lower faces $M^{n+q}\times LI^q$ due to the relations,
where $LI^q=\{(t_1,\ldots ,t_q)\in I^q; \textrm{ some } t_i=0\}$. So possible intersections are of
lower filtration. We have to find a map 
$$h:M^{n+q}\times I^q\times \Delta^n\times I\to M^{n+q}\times\Delta^{n+q}$$
over $M^{n+q}$ which is already defined on
$$M^{n+q}\times (LI^q\times \Delta^n\times I\cup I^q\times \partial(\Delta^n\times I)).$$
Since $LI^q$ is a strong deformation retract of $I^q$, the inclusion
$$LI^q\times \Delta^n\times I\cup I^q\times \partial(\Delta^n\times I)\subset I^q\times \Delta^n\times I$$
is an inclusion of a strong deformation retract. Hence $h$ exists.
\end{proof}

\textbf{Proof of Proposition \ref{4_8}}:
Since $M$ is well-pointed, the projection $p(M):\widetilde{B}M\to BM$ is a homotopy equivalence. If $h:X\to Y$ is a 
weak equivalence of semigroups, then $\widetilde{B}h:\widetilde{B}X\to 
\widetilde{B}Y$ is a based homotopy equivalence. Hence it suffices to show that
$$
\eta(BWM)\circ B\mu(WM)\circ p(WM)\circ \widetilde{B}\overline{\varepsilon}(WM)
\simeq  p(WM)\circ \widetilde{B}\overline{\varepsilon}(WM).$$
Now
 \begin{align*}
\eta(BWM) & \circ B\mu(WM)  \circ p(WM) \circ \widetilde{B}\bar{\epsilon}(WM) 
\\
=& \ev(BWM)  \circ B \varepsilon(\Omega'BWM) \circ BW \rho(WM) \circ B \nu(WM) \circ p(WM) 
                                                          \circ \widetilde{B}\bar{\varepsilon}(WM),
\intertext{since $\eta(BWM) = \ev(BWM)  \circ B \varepsilon(\Omega'BWM)$ and $\mu(WM) =W \rho(WM) \circ  \nu(WM) $,}
 \begin{split} 
=& \ev(BWM) \circ B \varepsilon(\Omega'BWM) \circ BW \rho(WM) \circ  p(WVP(EWM,WM))
 \\ & \circ \widetilde{B}\nu(WM)  \circ \widetilde{B}\bar{\varepsilon}(WM),
\end{split}
\intertext{by naturality of $p$,}
 \begin{split} =&  \ev(BWM) \circ B\rho(WM) \circ p(VP(EWM,WM))  \circ \widetilde{B}\varepsilon(VP(EWM,WM)) 
  \\    & \circ \widetilde{B}\nu(WM) \circ   \widetilde{B}\bar{\varepsilon}(WM),
\end{split}
\intertext{by naturality of $\varepsilon$,}
\begin{split}
=& \ev(BWM) \circ p(\Omega'BWM)\circ \widetilde{B}\rho(WM) \circ \widetilde{B}\varepsilon(VP(EWM,WM)) \\
    &                     \circ \widetilde{B}\nu(WM) \circ \widetilde{B}\bar{\varepsilon}(WM),
\end{split}
\intertext{by naturality of $p$ again,}
\begin{split} =& \ev(BWM) \circ p(\Omega'BWM)\circ \widetilde{B}\rho'(WM) \circ \widetilde{B}q(P(EWM,WM)) \\
  & \circ \widetilde{B}\varepsilon(VP(EWM,WM)) 
           \circ \widetilde{B}\nu(WM) \circ \widetilde{B}\bar{\varepsilon}(WM),
\end{split}
\intertext{by the definition of $\rho(WM)$,}
\begin{split}
=& \ev(BWM) \circ p(\Omega'BWM)\circ \widetilde{B}\rho'(WM) \circ \widetilde{B}\sigma_1(WM) \circ \widetilde{B}\nu(WM)\\                                                            &                \circ \widetilde{B}\bar{\varepsilon}(WM),
\end{split}
\intertext{by the definition of $\sigma_1(WM)$ from 4.12,}
\simeq & \ev(BWM) \circ p(\Omega'BWM)\circ\widetilde{B}\rho'(WM) \circ \widetilde{B}\zeta(WM),
\intertext{by Diagram 4.12,}
\simeq & p(WM) \circ \widetilde{B}\varepsilon(WM),
\end{align*}
by Lemma 4.14. \hfill\ensuremath{\Box}

\begin{rema}
If we use the Quillen model structure on $\Top$ rather than the Str{\o}m structure we can construct a homotopy
 unit $\mu(WM)$ and deduce
Proposition \ref{4_8} fairly easily from \cite[Thm. 7.3]{Fied} and its proof.
\end{rema}

\vspace{2ex}
The proof of the first part of \ref{4_4} needs some preparation. Let $\scJ$ denote the category of ordered sets
$[n]=\{0<1<\dots <n\}$ and order preserving injections, and let $\scJ\Top_0$ denote the category of all diagrams
$$X_\bullet: \scJ^{\op}\to \Top,\qquad [n]\mapsto X_n$$
such that $X_0$ is a single point, i.e. an object in $\scJ\Top_0$ is a reduced simplicial space without degeneracies.
Of lately, such an object is called a \textit{reduced semisimplicial space}. The usual fat topological realization functor
$$
\scJ\Top_0 \to \top,\qquad X_\bullet \mapsto \parallel X_\bullet\parallel
$$
has a right adjoint, the \textit{reduced singular functor}
$$
\Sing_\bullet^0:\top \to \scJ\Top_0, \qquad \Sing_n^0(Y)=\Top((\Delta^n,\Delta^n_0),(Y,\ast))
$$
where $\Delta^n_0$ is the $0$-skeleton of $\Delta^n$. The unit of this adjunction
$$
\tau_\bullet (X_\bullet): X_\bullet \to \Sing_\bullet^0\parallel X_\bullet \parallel
$$
sends $x\in X_n$ to the singular simplex 
$$
\Delta^n\xrightarrow{i_x}\coprod_k X_k\times \Delta^k\rightarrow \parallel X_\bullet \parallel
$$
where $i_x$ is the inclusion of the simplex $\{ x\}\times \Delta^n$. The counit
$$
\widehat{\ev}(Y): \parallel \Sing_\bullet^0(Y) \parallel\rightarrow Y
$$
is induced by the evaluation maps $\Sing_n^0(Y)\times \Delta^n\to Y$. The formula defining our evaluation map of
\ref{4_5} defines a natural semisimplicial map
$$
\alpha_\bullet(Y): N_\bullet \Omega'Y\rightarrow \Sing_\bullet^0 Y
$$
where $ N_\bullet \Omega'Y$ is the semisimplicial nerve of $\Omega'Y$. Let $v_0,\ldots,v_n$ denote the vertices
of $\Delta^n$ and let $L_n\subset \Delta^n$ denote the union of the $1$-simplexes $[v_{i-1},v_i]$, $i=1,\ldots,n$.
Then $L_n$ is a strong deformation retract of $\Delta^n$. The composite
$$
(\Omega'Y)^n=N_n(\Omega'Y)\xrightarrow{\alpha_n(Y)}\Sing_n^0(Y)\xrightarrow{r}\Top((L_n,\Delta^n_0),(Y,\ast))=(\Omega Y)^n,
$$
where $r$ is the restriction to $L_n$, is the map normalizing the loop lengths to $1$. In particular, $\alpha_n(Y)$ is
a homotopy equivalence inducing a homotopy equivalence $\parallel \alpha_\bullet(Y)\parallel$. Moreover, the diagram
\begin{leer}\label{4_10}
 $$\xymatrix{
\widetilde{B}\Omega'Y\ar[rr]^{p(\Omega'Y)}\ar[d]_{\parallel \alpha_\bullet(Y)\parallel} && B\Omega'Y\ar[d]^{\ev (Y)}\\
\parallel \Sing_\bullet^0(Y)\parallel\ar[rr]^(0.6){\widehat{\ev}(Y)}&& Y
}
$$
commutes.
\end{leer}

\begin{prop}\label{4_11}
\begin{enumerate}
\item[(1)] If $M$ is a grouplike well-pointed monoid, then $\rho'(M):P(EM,M)\to \Omega'BM$ and hence $\mu(WM):
WM\to W\Omega'BWM$ are weak equivalences.
\item[(2)] If $Y$ is a well-pointed path-connected Dold space (see Definition \ref{1_6a}), then
$\ev(Y): B\Omega'Y \to Y$ is a based homotopy equivalence, and, hence so is
$$\eta(Y) :BW\Omega'Y\xrightarrow{B\varepsilon(\Omega'Y)} B\Omega'Y\xrightarrow{\ev(Y)} Y.$$
\item[(3)] If $Y$ is a well-pointed space, then $\Omega'\ev(Y): \Omega'B\Omega'Y \to\Omega' Y$ is a weak equivalence. Hence so is 
$\Omega'\eta(Y):\Omega'BW\Omega'Y\to \Omega'Y$.
\item[(4)] If $M$ is a well-pointed monoid, then $B\mu(WM): BWM\to BW\Omega'BWM$ is a homotopy equivalence.
\end{enumerate}
\end{prop}
\begin{proof} (1)
 The diagram
$$\xymatrix{
M\ar[rr]^{\tau_1(N_\bullet M)}\ar[d]_{\zeta\circ \bar{\iota}(M)} &&\Omega\widetilde{B}M\ar[rr]^{\Omega p(M)} && 
\Omega BM\ar[d]^{i(BM)}\\
P(EM,M)\ar[rrrr]^{\rho'(M)} &&&&\Omega'BM
}
$$
commutes. Here $i(X):\Omega X\to \Omega'X$ is the inclusion and $\bar{\iota} (M):M\to \overline{W}M$ the section (see \ref{2_2}).
It is well known that $\tau_1(N_\bullet M)$ is a homotopy equivalence if $M$ is grouplike (e.g. see \cite{Segal}). Since
$p(M),\ i(BM)$, and $\zeta\circ \bar{\iota}(M)$ are homotopy equivalences in $\Top$, so is $\rho'(M)$.

(2) In the commutative diagram \ref{4_10} the map $p(\Omega'Y)$ is a homotopy equivalence because $\Omega'Y$ is well-pointed 
and $\widehat{\ev}(Y)$ is a homotopy equivalence by \cite[Prop. 5.6]{Schwam}.

(3) Consider the following commutative diagram in $\scJ\Top_0$
$$\xymatrix{
N_\bullet\Omega'Y\ar[rr]^(0.4){\tau_\bullet(N_\bullet \Omega'Y)}\ar[d]_{\alpha_\bullet(Y)}
&& \Sing_\bullet^0\widetilde{B}\Omega'Y\ar[d]^{\Sing_\bullet^0\parallel \alpha_\bullet(Y)\parallel}\\
\Sing_\bullet^0Y\ar[rr]^(0.4){\tau_\bullet(\Sing_\bullet^0Y)}\ar[rrd]_{\id} &&
\Sing_\bullet^0\parallel\Sing_\bullet^0Y \parallel \ar[d]^{\Sing_\bullet^0\widehat{\ev}(Y)}\\
&& \Sing_\bullet^0Y
}
$$
Restricting this diagram to degree $0$ we obtain a commutative diagram of spaces
$$\xymatrix{
\Omega'Y\ar[rr]^(0.4){\tau_1(N_\bullet \Omega'Y)}\ar[d]_{\alpha_1} && \Omega\widetilde{B}\Omega'Y
\ar[d]^{\Omega\parallel \alpha_\bullet(Y)\parallel}\ar[rr]^{\Omega p(\Omega'Y)}
&& \Omega B\Omega' Y \ar[d]^{\Omega \ev (Y)}\\
\Omega Y\ar[rr]^(0.35){\tau_1(\Sing_\bullet^0Y)} && \Omega\parallel\Sing_\bullet^0Y \parallel \ar[rr]^(0.6){\Omega\widehat{\ev}(Y)}
&& \Omega Y.
}
$$
Since $\Omega'Y $ is grouplike, $\tau_1(N_\bullet \Omega'Y)$ is a homotopy equivalence.
Since $Y$ and hence $\Omega'Y$ is well-pointed, $\Omega p(\Omega'Y)$ is a homotopy equivalence.
 Since $\alpha_1$ and $\parallel \alpha_\bullet(Y)\parallel$
are homotopy equivalences, $\tau_1(\Sing_\bullet^0Y)$ is one. Hence so is $\Omega\widehat{\ev}(Y)$ and hence also $\Omega \ev (Y)$, which
implies the result.

(4) Since $BWM$ is a well-pointed path-connected Dold space by \cite[Cor. 5.2]{Schwam} the statement follows from Part (2) and
Proposition \ref{4_8}.
\end{proof}

\begin{prop}\label{4_12}
 $W\Omega'\eta(X)\circ \mu(W\Omega'X)\simeq \id_{W\Omega'X}$ in  $\Mon$.
\end{prop}

\begin{proof}
It follows from Proposition \ref{4_8} and the homotopy naturality of $\mu$ and $\eta$ that the following diagram commutes up to homotopy.
$$\xymatrix{
W\Omega'X\ar[rr]^{\mu(W\Omega'X)}\ar[d]_{\mu(W\Omega'X)} && W\Omega'BW\Omega'X\ar[rr]^{W\Omega'\eta(X)}\ar[d]^{\mu(W\Omega'BW\Omega'X)}
&& W\Omega'X\ar[d]^{\mu(W\Omega'X)}\\
W\Omega'BW\Omega'X\ar[rr]^(0.45){W\Omega'B\mu(W\Omega'X)}\ar[rrd]_{\id} && W\Omega'BW\Omega'BW\Omega'X \ar[rr]^(0.55){W\Omega'BW\Omega'\eta(X)}
\ar[d]^{W\Omega'\eta(BW\Omega'X)} && W\Omega'BW\Omega'X\ar[d]^{W\Omega'\eta(X)}\\
&& W\Omega'BW\Omega'X\ar[rr]^{W\Omega'\eta(X)} && W\Omega'X
}
$$
We obtain
$$
W\Omega'\eta(X)\circ \mu(W\Omega'X)\circ W\Omega'\eta(X) \circ \mu(W\Omega'X)\simeq W\Omega'\eta(X)\circ \mu(W\Omega'X).
$$
Since $\Omega'X$ is grouplike $\mu(W\Omega'X)$ and $W\Omega'\eta(X)$ are weak equivalences by Proposition \ref{4_11}. By 
Proposition \ref{2_7} both homomorphisms have homotopy inverses in $\Mon$ so that
$$
W\Omega'\eta(X)\circ \mu(W\Omega'X)\simeq \id_{W\Omega'X}
$$
in $\Mon$.
\end{proof}

\section{Immediate consequences}

\textbf{The James Construction:}

The underlying space functor $U:(\Mon,\Monw)\to(\top,\topw)$ has a left adjoint
$$
J:(\top,\topw)\to(\Mon,\Monw)
$$
commonly called the \textit{James construction}, which associates with each based space $X$ the free based topological monoid on $X$.

\begin{prop}\label{5_1}
(James \cite{James}) For each path-connected based space there is a weak homotopy equivalence of spaces
$$
JX\simeq \Omega\Sigma X.
$$
\end{prop}

D. Puppe investigated the conditions which would imply for this weak homotopy equivalence to be a genuine homotopy equivalence.

\begin{prop}\label{5_2}
(Puppe \cite{DKP}): If $X$ is a well-pointed path-connected Dold space then there is a homotopy equivalence
$$
JX\simeq\Omega\Sigma X.
$$
\end{prop}

Consider the diagram of functors
$$
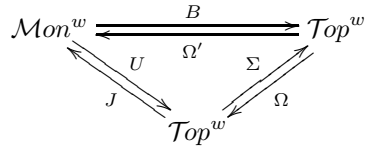
\xymatrix{
\Monw \ar@<0.4ex>[rr]^B \ar@<0.4ex>[dr]^U && \topw\ar@<0.4ex>[ll]^{\Omega'}\ar@<0.4ex>[dl]^\Omega
\\
& \topw \ar@<0.4ex>[ur]^\Sigma \ar@<0.4ex>[ul]^J
}
$$
All functors preserve weak equivalences. Hence they induce a diagram 
$$
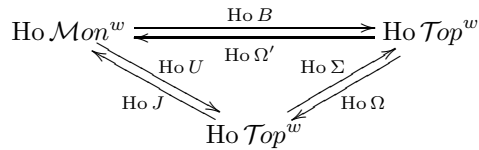
\xymatrix{
\Ho\Monw \ar@<0.4ex>[rr]^{\Ho B} \ar@<0.4ex>[dr]^{\Ho U} && \Ho\topw\ar@<0.4ex>[ll]^{\Ho \Omega'}\ar@<0.4ex>[dl]^{\Ho \Omega}
\\
& \Ho\topw \ar@<0.4ex>[ur]^{\Ho \Sigma} \ar@<0.4ex>[ul]^{\Ho J}
}
$$
consisting of adjoint pairs. Since the Moore loop space functor is naturally homotopy equivalent to the usual
loop space functor there is a natural transformation
$$
\tau(X):U\circ\Omega'(X)\to\Omega(X)
$$
which is a homotopy equivalence. Hence $\Ho\Omega$ and $\Ho U\circ\Ho\Omega'$ are naturally isomorphic. 
Since their left adjoints are unique up to natural
isomorphisms this implies that $\Ho B\circ \Ho J$ and $\Ho\Sigma$ are naturally isomorphic. We obtain

\begin{prop}\label{5_3}
For each $X\in\topw$ there is a homotopy equivalence 
$$
BJ(X)\simeq\Sigma(X)
$$
natural up to homotopy.\hfill$\square$
\end{prop}

We obtain Puppe's result by combining \ref{5_3} with another well-known result:

\begin{prop}\label{5_4}
 If $M$ is a well-pointed monoid whose underlying space is a Dold space and $\pi_0(M)$ is a group, then $M$ is grouplike \cite[(12.7)]{DKP}.
\end{prop}

\textbf{Proof of \ref{5_2}:}
If $X$ is a path-connected Dold space, so is $JX$. Hence $JX$ is grouplike and 
$\mu(WJX): WJX\to W\Omega'BWJX$ is a weak equivalence by \ref{4_11}, so that $\Omega'B\varepsilon(JX)\circ
 \varepsilon(\Omega'BJX)\circ \mu(WJX)\circ \iota(JX)
:JX\to\Omega'BJX $ is 
a homotopy equivalence. We have a sequence of homotopy
equivalences 
$$
JX\simeq \Omega'BJX\simeq\Omega BJX\simeq \Omega \Sigma X.$$

\vspace{2ex}
\textbf{Homotopical group completion:}
Homotopical group completion is the replacement of a monoid by a grouplike one having a universal property. We state our result for the
full subcategory $\Ho\Monw$ of $\Ho\Mon$ of well-pointed monoids. Since $q(M):VM\to M$ is a weak equivalence, $\Ho\Monw$ is equivalent
to $\Ho\Mon$ so that the corresponding statement for $\Ho\Mon$ follows.

\begin{prop}\label{5_5}
Let $M$ be a well-pointed monoid. The homotopy class of the homomorphism $\mu(WM):WM\to W\Omega'BWM$, considered as a morphism in 
\linebreak
$\Ho\Monw(M, \Omega'BWM)$,
is a group completion in the following sense: Given a diagram
$$
\xymatrix{
M \ar[rr]^(0.43){[\mu(WM)]} \ar[dr]_{[g]} && \Omega'BWM \ar@{-->}[dl]^{[\overline{g}]}
\\ &N
}
$$
in $\Ho\Monw$ with $N$ grouplike, there exists a unique morphism $[\overline{g}]:\Omega'BWM\to N$ making the diagram commute.
\end{prop}

\begin{proof}
Consider the homotopy commutative diagram in $\Monw$
$$
\xymatrix{
WM \ar[rr]^(0.4){\mu(WM)} \ar[d]_g
&& W\Omega'BWM \ar[d]^{W\Omega'Bg}
\\
WN \ar[rr]^(0.4){\mu(WN)}
&& W\Omega'BWN 
}
$$
Since $N$ is well-pointed and grouplike $\mu(WN)$ is homotopy invertible in $\Monw$ by \ref{4_11}. We choose a homotopy inverse
$h:W\Omega'BWN\to WN$ and define $\overline{g}=h\circ W\Omega'Bg$. Then $\overline{g}\circ\mu(WM)\simeq g$ in $\Monw$.

 For the uniqueness of $[\overline{g}]$ suppose there
is a homomorphism $g': W\Omega'BM \longrightarrow  WN$ such that $h \circ\mu(WM)\simeq g$. 
Put $j_1=\mu(WN)\circ g'$ and $j_2=W\Omega'Bg$.
It suffices to show
that $j_1\simeq j_2$ in $\Monw$. Since 
$Bj_1\circ B\mu(WM)\simeq \mu(WN)\circ Bg\simeq Bj_2\circ B\mu(WM)$
and $B\mu(WM)$ is a homotopy equivalence by \ref{4_11}, we obtain $Bj_1\simeq Bj_2$. Since $\mu(W\Omega'BM)$ and
$\mu(W\Omega'BN)$ are homotopy equivalences in $\Monw$ by \ref{4_11} and $\mu$ is natural up to homotopy the following diagram is homotopy 
commutative and establishes the result:
$$\xymatrix{
W\Omega'BM\ar[rr]^(0.43){\mu(W\Omega'BM)}\ar[d]_{j_k} &&  W\Omega'B W\Omega'BM\ar[d]^{W\Omega'Bj_k}\\
W\Omega'BN\ar[rr]^(0.43){\mu(W\Omega'BN)} && W\Omega'B W\Omega'BN
}
$$
\end{proof}

\textbf{Dold spaces and grouplike monoids}

For details on Dold spaces see \cite{Schwam}. We restrict our attention to the well-pointed case. Using the whiskering process it is easy to 
extend our results to the general case.

Let $\topw_{Dold}\subset \topw$ denote the full subcategory of well-pointed path-connected Dold spaces. Since $BM$ is in $\topw_{Dold}$ for
any well-pointed monoid by \cite[Cor. 5.2]{Schwam}, the classifying space functor restricts to a functor
$$B:\Monw\to \topw_{Dold}.$$
Let $\Monw_{group}\subset \Monw$ denote the full subcategory of grouplike well-pointed monoids. Then Proposition \ref{4_11} implies

\begin{theo}\label{5_6}
The functors
$$\BwH:\scH\Monw_{group}\leftrightarrows \scH\topw_{Dold}:\OwH$$
define an equivalence up to homotopy of categories, i.e. the natural transformations up to homotopy $\mu: \Id \to \OwH\circ \BwH$ and
$\eta: \BwH\circ \OwH\to \Id$ take values in
homotopy equivalences. In particular,
$$\Ho B:\Ho \Monw_{group}\leftrightarrows \Ho \topw_{Dold}:\Ho \Omega'$$
define an equivalence of categories.\hfill\ensuremath{\Box}
\end{theo}

The second part is a slight extension of a well-known result (e.g. see \cite[Section 4]{Brink}).

The following two propositions extend and strengthen results of Fuchs \cite[Satz 7.7]{Fuchs}. 

The diagram
$$
\xymatrix{
\Monw(WM,WN) \ar[rr]^{\mu(WN)_\ast} \ar[rrd]_B && \Monw(WM,W\Omega'BWN)\ar[d]^{\lambda(WM,BWN)}\\
&& \topw(BWM,BWN)
}
$$
commutes up to homotopy because by \ref{4_8}
$$\begin{array}{rcl}
(\lambda(WM,BWN)\circ \mu(WN)_\ast)(f) &= & \eta(BWN)\circ B' \circ \mu(WN)_\ast)(f)\\
& = & \eta(BWN)\circ B\mu(WN)\circ Bf\\
& \simeq & Bf \quad\textrm{continuously in } f
\end{array}
$$
with $B': \Monw(WM,W\Omega'BWN)\to \topw(BWM,BW\Omega'BWN)$. If $N$ is 
grouplike $\mu(WN)$ is a homotopy equivalence in $\Monw$, and we obtain 

\begin{prop}\label{5_7}
 If $N$ is a well-pointed grouplike monoid then 
$$B: \Monw(WM,WN)\to \topw(BWM,BWN)$$
is a homotopy equivalence.\hfill\ensuremath{\Box}
\end{prop}

Since $\eta(X):BW\Omega'X\to X$ is a natural transformation the following diagram commutes
$$
\xymatrix{
\Monw(W\Omega'X,W\Omega'Y)\ar[d]^{B} & \topw(X,Y)\ar[l]_(.4){W\Omega'} \ar[d]_{\eta(X)^\ast}\\
\topw(BW\Omega'X,BW\Omega'Y) \ar[r]^(0.55){\eta(Y)_\ast} & \topw(BW\Omega'X,Y).
}
$$
Since $W\Omega'Y$ is grouplike the map $B$ is a homotopy equivalence by \ref{5_7}. Since $\eta(Y)_\ast\circ B=\lambda(W\Omega'X,Y)$ the
map $\eta(Y)_\ast: \topw(BW\Omega'X,BW\Omega'Y)\to \topw(BW\Omega'X,Y)$ is a homotopy equivalence.
If $X$ is a well-pointed path-connected Dold space $\eta(X)$  is a based homotopy equivalence by \ref{4_11}. We obtain

\begin{prop}\label{5_8} 
 If $X$ is a well-pointed path-connected Dold space then $W\Omega': \topw(X,Y)\to\Monw(W\Omega'X,W\Omega'Y)$ is a homotopy equivalence.
\hfill\ensuremath{\Box}
\end{prop}

\textbf{Homotopy homomorphisms and unitary homotopy homomorphisms}

\begin{prop}\label{5_9}
 Let $M$ and $N$ be well-pointed monoids and $N$ be grouplike. Then $\varepsilon'(M):\overline{W}M\to WM$ induces a
homotopy equivalence
$$\Mon(WM,N)\to \Sgp(\overline{W}M,N).$$
\end{prop}

\begin{proof}
 By \ref{2_4} we may replace $N$ by $WN$. Since $\Sgp(\overline{W}M,WN)$ is naturally homeomorphic to $\Mon(W(M_+),WN)$ by \ref{2_3b}
it suffices to show that the counit $\kappa(M):M_+\to M$ induces a homotopy equivalence
$$\kappa(M)^\ast: \Mon(WM,WN)\to \Mon(W(M_+),WN)$$
The diagram
$$\xymatrix{
\Mon(WM,WN)\ar[rr]^B\ar[d]_{\kappa(M)^\ast} && \top(BWM,BWN)\ar[d]^{B\kappa(M)^\ast}\\
\Mon(W(M_+),WN)\ar[rr]^B && \top(BW(M_+),BWN)
}
$$
commutes. By \ref{5_7} the maps $B$ are homotopy equivalences, and by \ref{4_1a} the map $B\kappa(M)^\ast$ is
a homotopy equivalence. Hence so is $\kappa(M)^\ast$.
\end{proof}

\begin{rema}\label{5_10}
 In general we cannot expect that $\varepsilon'(M):\overline{W}M\to WM$ induces a
homotopy equivalence. E.g. it can happen that a homomorphism $\overline{W}M\to N$ does not map $(e_M)$ into the path-component of $e_N$
so that there is no chance to homotop it into a homomorphism $WM\to N$. 
\end{rema}

\begin{prop}\label{5_11}
 If $M$ is a well-pointed monoid then $Wq(M):WVM\to WM$ is a homotopy equivalence in $\Monw$ by \ref{2_7} inducing a homotopy
equivalence
$$\Mon(WM,N)\to \Mon(WVM,N).$$
\end{prop}

\section{Diagrams of monoids}
We want to show that the homotopy adjunction of Theorem \ref{4_3a} lifts to diagram categories. This is not
evident: since the unit of our homotopy adjunction is only natural up to homotopy it does not lift to
diagrams.

Let $\scM$ be a cocomplete $\Top$-enriched tensored category with a class $\mathscr{W}$ of weak equivalences containing the
homotopy equivalences. We assume that $\scM$ has a strong cofibrant replacement functor $(Q_M,\varepsilon_M)$. We use
 $\otimes$ for the tensor in $\scM$ and $Q$ for $Q_M$ as long as there is no ambiguity.

\begin{defi}\label{6_1} Let $\mathcal{C}$ be a small indexing category.
A morphism $f:D_1\to D_2$ of $\mathcal{C}$-diagrams in $\scM$ is called a \textit{weak equivalence} if it is 
objectwise a weak equivalence in $\scM$. 
We denote the class of weak equivalences in $\scM^\mathcal{C}$ by $\mathscr{W}^\mathcal{C}$.
\end{defi}

Our first aim is to show that $\Mc$ admits a strong cofibrant replacement functor in order to make additional
applications of Proposition \ref{3_3}. Therefore we proceed as in \ref{2_4} and \ref{2_2}.

We define a $\mathcal{C}\times\mathcal{C}^{\op}$-diagram $B(\mathcal{C},\mathcal{C},\mathcal{C})$ in $\Top$ as follows:
$$
B(\mathcal{C},\mathcal{C},\mathcal{C})(b,a)=B(\mathcal{C}(-,b),\mathcal{C},\mathcal{C}(a,-))
$$
where the right side is the 2-sided bar construction of \ref{4_1}. 

The $\mathcal{C}\times\mathcal{C}^{\op}$ structure on $B_n(\mathcal{C},\mathcal{C},\mathcal{C})$ is given by
$$
(g, h)\cdot(f_0,\ldots,f_{n+1})=(g\circ f_0, f_1,\ldots, f_n, f_{n+1}\circ h)
$$

Analogously we define a $\mathcal{C}^{\op}$-diagram $B(\ast,\mathcal{C},\mathcal{C})$ in $\Top$, where $\ast$ 
denotes the constant $\mathcal{C}^{\op}$-diagram on a single point.

\begin{lem}\label{6_2}
Let $X$ and $Y$ be $\mathcal{C}\times\mathcal{C}^{\op}$-diagrams in $\Top$, let $p:X\to Y$ be 
a map of diagrams which is objectwise a 
homotopy equivalence. 
 Then $p$ induces a homotopy equivalence
$$
p_\ast:\Top^{\mathcal{C}\times\mathcal{C}^{\op}}(B(\mathcal{C},\mathcal{C},\mathcal{C}),X)\to 
\Top^{\mathcal{C}\times\mathcal{C}^{\op}}(B(\mathcal{C},\mathcal{C},\mathcal{C}),Y)
$$
in $\Top$.
\end{lem}
\begin{proof}
We apply the HELP-Lemma. So given a diagram
$$\xymatrix{
K\ar[rr]^(0.3){\bar{f}_K} \ar[d]_i && \Top^{\mathcal{C}\times\mathcal{C}^{\op}}(B(\mathcal{C},\mathcal{C},\mathcal{C}),X) \ar[d]^{p_\ast}\\
L\ar[rr]^(0.3){\bar{g}} && \Top^{\mathcal{C}\times\mathcal{C}^{\op}}(B(\mathcal{C},\mathcal{C},\mathcal{C}),Y)
}
$$
which commutes up to a homotopy $\bar{h}_{K,t}:\bar{g}\circ i\simeq p_\ast\circ \bar{f}_K$, where $i$ is
 a closed cofibration, we have to construct
extensions
$$
\begin{array}{rcl}
\bar{f}: L & \rightarrow & \Top^{\mathcal{C}\times\mathcal{C}^{\op}}(B(\mathcal{C},\mathcal{C},\mathcal{C}),X)\\
\bar{h}_t: L& \rightarrow & \Top^{\mathcal{C}\times\mathcal{C}^{\op}}(B(\mathcal{C},\mathcal{C},\mathcal{C}),Y)
\end{array}
$$
of $\bar{f}_K$ respectively $\bar{h}_{K,t}$ such that
$\bar{h}_t: \bar{g}\simeq p_\ast\circ \bar{f}$.

Taking adjoints the above diagram translates to the following diagram of $\mathcal{C}\times \mathcal{C}^{\op}$-spaces
$$\xymatrix{
K\times B(\mathcal{C},\mathcal{C},\mathcal{C})\ar[rr]^(0.65){f'}\ar[d]_{i\times \id} && X\ar[d]^{p} \\
L\times B(\mathcal{C},\mathcal{C},\mathcal{C})\ar[rr]^(0.65){g} && Y
}
$$
which commutes up to a homotopy $h'_t:g\circ (i\times \id)\simeq p\circ f'$ in $\Top^{\scC\times\scC^{\op}}$, and it suffices to construct
 extensions $f: L\times B(\mathcal{C},\mathcal{C},\mathcal{C})\to X$ of $f'$ and $h_t:L\times B(\mathcal{C},\mathcal{C},\mathcal{C})\to Y$
of $h'_t$ such that $h_t: g\simeq p \circ f$ in $\Top^{\scC\times\scC^{\op}}$.

 We construct
these extensions  by induction on the natural filtration $F_n$ of $L\times B(\mathcal{C},\mathcal{C},\mathcal{C})$ induced by 
the realization of the simplicial set
$B_\bullet(\mathcal{C},\mathcal{C},\mathcal{C})$. We start with $F_0=\coprod_{a,b,c} L\times\scC(c,b)\times \scC(a,c)$. The diagram
$$\xymatrix{
K\times \{(\id_c,\id_c)\} \ar[d] \ar[rr]^(0.6){f'} && X(c,c)\ar[d]^{p}\\
L\times \{(\id_c,\id_c)\ar[rr]^(0.6){g} && Y(c,c)
}
$$
commutes up to a homotopy given by $h_t'$. Since $p: X(c,c)\to Y(c,c)$ is a homotopy equivalence and $K\to L$ is a closed cofibration
the required extensions exist by the HELP-Lemma.
We extend $f$ over all of $F_0$ by $f(l,j_0,j_1)=X(j_0,j_1)\circ f(l,\id,\id)$ and analogously for $h_t$. \\
Now suppose that $f$ and $h_t$ have been defined on $F_{n-1}$.
We obtain $F_n$ from $F_{n-1}$ by attaching spaces $L\times (j_0,\ldots ,j_{n+1})\times \Delta^n$
 along $L\times (j_0,\ldots ,j_{n+1})\times \partial\Delta^n$.
Here the $j_k$ are morphisms in $\scC$ such that the composition 
$$j_0\circ \ldots j_{n+1}:a\to c_n\to \ldots \ c_0\to b$$ is defined and $j_1,\ldots, j_n$ are not identities. 
Hence the extension $f$ and the homotopy $h_t$ are already defined on

$D(L\times (j_0,j_1,\ldots,j_n,j_{n+1})\times \Delta^n)\\
{} \qquad\qquad = K\times (j_0,j_1,\ldots,j_n,j_{n+1})\times \Delta^n \cup L\times 
(j_0,j_1,\ldots,j_n,j_{n+1})\times \partial\Delta^n.$ 

We apply the HELP-Lemma
to the homotopy commutative diagram
$$\xymatrix{
 D(L\times (\id_{c_0},j_1,\ldots,j_n,\id_{c_n})\times \Delta^n) \ar[rr]^(0.7){f''}\ar[d] &&
X(c_0,c_n)\ar[d]^{p}\\
L\times (\id_{c_0},j_1,\ldots,j_n,\id_{c_n})\times \Delta^n \ar[rr]^(0.7)g &&
Y(c_0,c_n)
}
$$
where $f''$ and the commuting homotopy are given by the already defined extensions. Since ${p}$ is objectwise a homotopy equivalence 
and the inclusion
$$
D(L\times (\id_{c_0},j_1,\ldots,j_n,\id_{c_n})\times \Delta^n)\subset L\times (\id_{c_0},j_1,\ldots,j_n,\id_{c_n})\times \Delta^n
$$
is a closed cofibration the required
extensions exist. We extend our maps to maps of diagrams as in the $F_0$-case.
\end{proof}

Let $D$ be a $\mathcal{C}$-diagram in $\scM$ and $X$ a $\mathcal{C}^{\op}$-diagram in $\Top$. 
We define $X\otimes_\mathcal{C} D$ to be the coequalizer in $\scM$ of
$$
\xymatrix{
\coprod\limits_{f\in \textrm{ mor } \mathcal{C}} X(\textrm{target} (f))\otimes D(\textrm{source}(f))
\ar@<1.5ex>[r]^(0.65)\alpha\ar@<0.5ex>[r]_(0.65)\beta &
\coprod\limits_{a\in\rm{ob}\mathcal{C}} X(a)\otimes D(a)
}
$$
where for $f:a\to b$ in $\scC$ the $f$-summand $X(b)\otimes D(a)$ is mapped as follows
$$
\begin{array}{rcrrcl}
\alpha &=& X(f)\otimes\id: & X(b)\otimes D(a) &\longrightarrow & X(a)\otimes D(a)\\
\beta &=& \id\otimes D(f): & X(b)\otimes D(a) &\longrightarrow & X(b)\otimes D(b)
\end{array}
$$
We define a functor
$$
 R:\Mc \rightarrow \Mc ,\qquad D\mapsto B(\mathcal{C},\mathcal{C},\mathcal{C})\otimes_\mathcal{C} QD
$$
where $B(\mathcal{C},\mathcal{C},\mathcal{C})\otimes_\mathcal{C}D$
is the $\mathcal{C}$-diagram 
$$
a\mapsto B(\mathcal{C}(-,a),\mathcal{C},\mathcal{C})\otimes_\mathcal{C} D
$$
in $\scM$. 

\begin{prop}\label{6_3}
Let $D_0,\  D_1$, and $D_2$  be $\mathcal{C}$ diagrams in $\scM$, let $p:D_1\to D_2$ be a weak equivalence
in $\Mc$ and $q:A_1\to A_2$ a weak equivalence in $\scM$.
 Then $p$ and $q$ induce homotopy equivalences
$$\begin{array}{rcl} 
p_\ast:\Mc (R D_0,D_1) &\rightarrow & \Mc(R D_0,D_2)\\
q_\ast: \scM(B(\ast,\mathcal{C},\mathcal{C})\otimes_\mathcal{C}QD_0, A_1) & \rightarrow &\scM(B(\ast,\mathcal{C},\mathcal{C})\otimes_\mathcal{C}QD_0, A_2)
\end{array}
$$
in $\Top$.
\end{prop}
\begin{proof}
 Since $\Mc (R D_0,D_i)\cong \Top^{\scC\times \scC^{\op}}(B(\mathcal{C},\mathcal{C},\mathcal{C}), \scM(QD_0,D_i))$
it follows from Lemma \ref{6_2} with $X(b,a)=\scM (QD_0(a),D_1(b))$ and 
 $Y(b,a)=\scM (QD_0(a),D_2(b))$ that $p_\ast$ is a homotopy equivalence. 

There is a sequence of natural homeomorphisms
$$
\begin{array}{rcl}
 \scM(B(\ast,\mathcal{C},\mathcal{C})\otimes_\mathcal{C}QD_0, A_i) & \cong & \Top^{\scC^{\op}}(B(\ast,\mathcal{C},\mathcal{C}),\scM(QD_0,A_i)\\
& \cong & \Top^{\scC^{\op}}(\colim_{\scC} B(\scC ,\mathcal{C},\mathcal{C}),\scM(QD_0,A_i)\\
& \cong & \Top^{\scC\times\scC^{\op}}(B(\scC ,\mathcal{C},\mathcal{C}),\scM (QD_0,\cons A_i)
\end{array}
$$
where $\cons A_i$ are the constant $\scC$-diagrams on $A_i$. As in the first part, it follows that $q_\ast$ is a homotopy equivalence.
\end{proof}

 Let $\mathcal{C}_\bullet$ denote the $\mathcal{C}\times\mathcal{C}^{\op}$-diagram 
of simplical sets sending $(b,a)$ to the constant simplicial set $\mathcal{C}(a,b)$. The maps
$$
\begin{array}{rrcl}
\delta_n: & B_n(\mathcal{C},\mathcal{C},\mathcal{C}) (b,a)=B_n(\scC(-,b),\scC,\scC(a,-)) &\longrightarrow & \mathcal{C}(a,b)\\
& (f_0,\ldots,f_{n+1}) &\longmapsto & f_0\circ\ldots\circ f_{n+1}
\end{array}
$$
define a simplicial map $B_\bullet(\mathcal{C},\mathcal{C},\mathcal{C})\to \mathcal{C}_\bullet$. Let $\delta:B(\mathcal{C},\mathcal{C},
\mathcal{C})\to\mathcal{C}$ be its realization.

\begin{prop}\label{6_4}
$\delta(D)=\delta\otimes_\mathcal{C}\id_D:B(\mathcal{C},\mathcal{C},\mathcal{C})\otimes_\mathcal{C}D\to\mathcal{C} \otimes_\mathcal{C}D\cong D$ is 
objectwise a homotopy equivalence in $\scM$ and hence a weak equivalence in $\Mc$.
\end{prop}
The proposition is an immediate consequence of the following Lemma:

\begin{lem}\label{6_5}
For each object $b\in \scC$ the map $\varepsilon: B(\scC,\scC,\scC)(-,b)\to \scC(-,b)$ is a homotopy equivalence in the category $\Top^{\scC^{\op}}$. 
\end{lem}
\begin{proof}
 For $a\in \scC$ let $\scX_a$ denote the category whose objects are diagrams $a\xrightarrow{j_1}
 c\xrightarrow{j_0} b$ and whose morphisms from
this object to $a\xrightarrow{j'_1} c'\xrightarrow{j'_0} b$ are morphisms $h:c\to c'$ in $\scC$ making the diagram
$$
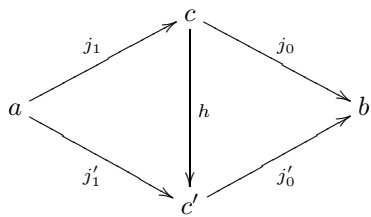
\xymatrix{
&& c\ar[dd]^h\ar[rrd]^{j_0} &&\\
a\ar[rru]^{j_1}\ar[rrd]_{j'_1} &&&& b\\
&& c'\ar[rru]_{j'_0}&&
}
$$
commute. Let $\scC(a,b)$ stand for the discrete category whose object set is $\scC(a,b)$. Then 
$$\varepsilon_a: \scX_a\to \scC(a,b),\qquad (a\xrightarrow{j_1} c\xrightarrow{j_0} b)\mapsto (j_0\circ j_1:a\to b)
$$
defines a functor which has the section
$$s_a:\scC(a,b)\to \scX_a,\qquad j\mapsto (a\xrightarrow{j} b\xrightarrow{\id} b).$$
There is a natural transformation $\tau_a:\Id_{\scX_a}\to s_a\circ \varepsilon_a$ defined by the diagram
$$
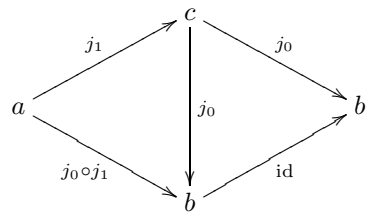
\xymatrix{
&& c\ar[dd]^{j_0}\ar[rrd]^{j_0} &&\\
a\ar[rru]^{j_1}\ar[rrd]_{j_0\circ j_1} &&&& b\\
&& b\ar[rru]_{\id}&&
}
$$
So $\varepsilon_a$ induces a homotopy equivalence of the classifying spaces. Now $B(\scX_a)=B(\scC,\scC,\scC(a,b))$ and $B(\scC(a,b))=\scC(a,b)$.
Moreover all data are natural with respect to $a\in \scC^{\op}$. Hence we obtain the required result.
\end{proof}

When we combine \ref{6_3} and \ref{6_4} we obtain the following corollary.

\begin{coro}\label{6_6}
$R:\Mc  \rightarrow \Mc$ together with $\epsilon=\delta\otimes_{\scC}\varepsilon_M: R\to \Id$ is a strong
cofibrant replacement functor.
\end{coro}

Let $\scN$ be another cocomplete $\Top$-enriched tensored category with a class of weak equivalences containing the homotopy equivalences
and a strong cofibrant replacement functor $(Q_N,\varepsilon_N)$. 

\begin{theo}\label{6_7} Let
 $$ F:\scM \leftrightarrows \scN:G$$
be continuous functors inducing a natural homotopy equivalence
$$
\lambda(Q_MA,Q_NX): \scM(Q_MA,Q_MGQ_NX)\to \scN(Q_NFQ_MA,Q_NX)
$$ 
so that
$$
F^{\scH}:\scH\scM \leftrightarrows \scH\scN:G^{\scH}
$$
is a conatural adjunction up to homotopy. Then
$$
(F^{\scC})^{\scH}: \scH(\Mc)\leftrightarrows \scH(\Nc):(G^{\scC})^{\scH}
$$
is an  adjunction up to homotopy, and hence
$$ \Ho (F^{\scC}):\Ho (\scM^{\scC}) \leftrightarrows \Ho (\scN^{\scC}):\Ho (G^{\scC})$$
a genuine adjunction.
\end{theo}

\begin{proof} For diagrams $D:\scC\to \scM$ and $Z:\scC\to \scN$ we have a sequence of natural maps
 $$
\xymatrix{
\Top^{\mathcal{C}\times\mathcal{C}^{\op}}(B(\mathcal{C},\mathcal{C},\mathcal{C}),\scN(Q_NFQ_MD,Q_NZ))\ar[r]^(0.63)\cong &
\Nc(R_NFQ_MD,Q_NZ)\ar[d]^{R_NF\delta(Q_MD)^\ast}\\
\Top^{\mathcal{C}\times\mathcal{C}^{\op}}(B(\mathcal{C},\mathcal{C},\mathcal{C}),\scM(Q_MD,Q_MGQ_NZ))
\ar[u]^{\lambda(Q_MD,Q_NZ)_\ast} & \Nc(R_NFR_MD,Q_NZ)\\
\Mc(R_MD,Q_MGQ_NZ)\ar[u]^\cong &   \Nc(R_NFR_MD,R_NZ)\ar[u]_{\delta(Q_NZ)_\ast}  \\
\Mc(R_MD,R_MGQ_NZ)\ar[u]^{\delta(Q_MGQ_NZ)_\ast} & \\
\Mc(R_MD,R_MGR_NZ)\ar[u]^{R_MG\delta(Q_NZ)_\ast} & \\
}
$$
By assumption $\lambda(Q_MD,Q_NZ)$ is a homotopy equivalence. Since $\delta(D)$ is objectwise a homotopy equivalence
 and since continuous functors preserve homotopy equivalences, $R_MG\delta(Q_NZ)$ and
$R_NF\delta(Q_MD)$ are homotopy equivalences in $\Mc$ by \ref{6_3} so that $R_MG\delta(Q_NZ)_\ast$ and $R_NF\delta(Q_MD)^\ast$ are
 homotopy equivalences in $\Top$,
 and $\delta(Q_NZ)_\ast$  and $\delta(Q_MGQ_NZ)_\ast$ are
homotopy equivalences in $\Top$ by \ref{6_3}.
\end{proof}

\begin{leer}\label{6_8}
\textbf{Addendum:}
The last natural map in the proof of the theorem points in the wrong direction. So we cannot conclude that 
$(F^{\scC})^{\scH}$ and $(G^{\scC})^{\scH}$ are a conatural homotopy adjoint pair. 
$$\eta(Q_NY)=\lambda(Q_MGQ_NY,Q_NY)(\id_{Q_MGQ_NY}):Q_NFQ_MGQ_NY\to Q_NY$$
is natural with respect to morphisms $f:Q_NY_1\to Q_NY_2$ in $\scN$. If $\eta$ extends to a natural map
$\eta(Y):Q_NFQ_MGY\to Y$ for all $Y\in\scN$ or at least for all $Y$ of the form $Y=R_NY'$
  we obtain a natural map $\lambda^{\scC}(R_MD,R_NZ)$ defined by
$$\xymatrix{
\Mc(R_MD,R_MGR_NZ) \ar[rr]^(0.45){R_NF}\ar[rrd]_{\lambda^{\scC}(R_MD,R_NZ)}   && \Mc(R_NFR_MD,R_NFR_MGR_NZ)\ar[d]^{\eta(R_NZ)} \\
&& \Nc(R_NFR_MD,R_NZ)
}$$
which makes the diagram of the proof of the theorem commute so that $(F^{\scC})^{\scH}$ and $(G^{\scC})^{\scH}$ are a
conatural homotopy adjoint pair.  
\end{leer}

For use in the next proposition we note

\begin{lem}\label{6_9} Let $D: \scC\to \Monw$ be a diagram of well-pointed monoids. Then
$B(\ast,\mathcal{C},\mathcal{C})\otimes_\mathcal{C} D$ is a well-pointed space, and
$B(\mathcal{C},\mathcal{C},\mathcal{C})\otimes_\mathcal{C} D$ and 
\linebreak
$B(\mathcal{C},\mathcal{C},\mathcal{C})\otimes_\mathcal{C} WD$ 
are diagrams of well-pointed monoids.
\end{lem}
\begin{proof}
The first part holds by \cite[Prop. 7.8]{PRV}. The second and third statement follow by the argument used in \cite[Prop. 7.8]{PRV}.
\end{proof}

>From \ref{6_6} and \ref{6_9} we obtain

\begin{prop}\label{6_10} With the choices of weak equivalences $\mathscr{W}$ as in \ref{3_4} the functors
$$
 \begin{array}{rcll}
\Monc & \rightarrow & \Monc & \qquad  D\mapsto B(\mathcal{C},\mathcal{C},\mathcal{C})\otimes_\mathcal{C} WVD\\
 ( \Monw)^{\scC}&\rightarrow &( \Monw)^{\scC}&\qquad  D\mapsto B(\mathcal{C},\mathcal{C},\mathcal{C})\otimes_\mathcal{C} WD\\
\Sgp^{\scC} & \rightarrow & \Sgp^{\scC} & \qquad  D\mapsto B(\mathcal{C},\mathcal{C},\mathcal{C})\otimes_\mathcal{C} \overline{W}D\\
(\top)^{\scC} & \rightarrow & (\top)^{\scC} & \qquad  D\mapsto B(\mathcal{C},\mathcal{C},\mathcal{C})_+\wedge_\mathcal{C} V^tD\\
(\topw)^{\scC} & \rightarrow & (\topw)^{\scC} & \qquad  D\mapsto B(\mathcal{C},\mathcal{C},\mathcal{C})_+\wedge_\mathcal{C} D
\end{array}
$$
together with the corresponding natural transformations $\epsilon$
are strong cofibrant replacement functors with respect to the weak equivalences in $\mathscr{W}^{\scC}$. In particular, the localizations
of these categories with respect to $\mathscr{W}^{\scC}$ exist. (Recall that $K_+\wedge X$ is the tensor over $\Top$ in $\top$.)
\end{prop}

Since Addendum \ref{6_8} applies to our situation in Section 4 we obtain

\begin{theo}\label{6_11}
 The homotopy adjunctions of Theorems \ref{4_3} and \ref{4_3a} lift to conatural homotopy adjunctions
$$
\xymatrix{
(B^{\scC})^{\scH}:\mathcal{H}(\Monw)^{\scC} \ar@<0.5ex>[r] & (\topw)^{\scC}:({\Omega'}^{\scC})^{\scH} \ar@<0.5ex>[l]
}
$$
and
$$
\xymatrix{
(B^{\scC})^{\scH}:\mathcal{H}\Mon^{\scC} \ar@<0.5ex>[r] & {\top}^{\scC}: ({\Omega'}^{\scC})^{\scH}. \ar@<0.5ex>[l]
}
$$
There are natural adjunction homotopy equivalences 
$$
\begin{array}{rcl}
\lambda(RD,QZ):(\Monw)^{\scC}(RD,R\Omega'QZ)&\rightarrow& (\topw)^{\scC}(QBRD,QZ)\\
\lambda(RVD,QV^tZ):\Mon^{\scC}(RVD,RV\Omega'QV^tZ)&\rightarrow& {\top}^{\scC}(QV^tBRVD,QV^tZ)
\end{array}
$$
in $\Top$, where $(R,\epsilon)$ and $(Q,\epsilon^t)$ are the cofibrant replacement functors in 
$(\Monw)^{\scC}$ respectively $(\topw)^{\scC}$ of \ref{6_10}. Hence
$$\Ho B^{\scC}: \Ho (\Monw)^{\scC}\leftrightarrows \Ho (\topw)^{\scC}:\Ho \Omega'^{\scC}$$
and
$$\Ho B^{\scC}: \Ho (\Mon)^{\scC}\leftrightarrows \Ho (\top)^{\scC}:\Ho \Omega'^{\scC}$$
are genuine adjunctions.
\end{theo}

\begin{theo}\label{6_12}
 Let $\scM$ be as above. Then the adjoint pair of functors
$$
\colim : \Mc \leftrightarrows \scM:\cons
$$
induces a conatural adjunction up to homotopy
$$
\colimH :\scH \Mc \leftrightarrows \scH\scM :\constH.
$$
Hence we obtain a genuine adjunction
$$
\Ho \colim : \Ho\Mc \leftrightarrows \Ho\scM:\Ho\cons
$$
\end{theo}

\begin{proof}
We have the following sequence of natural homotopy equivalences and homeomorphisms from
$\scH\Mc(D,\constH A)=\Mc(RD,R(\cons QA))$ to 
$\scH\scM (\colimH D,A)$
$=\scM (Q(\colim RD),QA)$:
$$
\begin{array}{lrcl}
 (1) & \Mc(RD,R(\cons QA)) &\xrightarrow{\epsilon(\cons QA)_\ast} & \Mc(RD, \cons QA)\\
(2) & &\stackrel{\cong}{\longrightarrow} & \scM(\colim RD, QA)\\
(3) && \xrightarrow{\varepsilon_M(\colim RD)^\ast} & \scM(Q(\colim QD),  QA).
\end{array}
$$
The first map is a homotopy equivalence by \ref{6_3}, the second one is the adjunction homeomorphism, and the third one is a homotopy equivalence,
because $\varepsilon_M(\colim RD): Q\colim RD \to \colim RD$ is a homotopy equivalence in $\scM$ by \ref{6_3}.
\end{proof}

\begin{defi}\label{6_13} Let $\scM$ be a cocomplete $\Top$-enriched tensored category with a class
$\mathscr{W}$ of weak equivalences containing the homotopy equivalences and equipped with a strong cofibrant replacement
functor $(Q,\varepsilon)$. Then the
\textit{homotopy colimit functor} $\hocolim: \Mc\to\scM$ is defined by 
$$\hocolim D=  \colim RD= B(\ast,\scC,\scC)\otimes_{\scC} QD.$$
\end{defi}

\begin{rema}\label{6_14} In the literature one often finds the homotopy colimit defined by
 $\hocolim D =B(\ast,\scC,\scC)\otimes_{\scC} D$ (e.g. see \cite[18.1.1]{Hirsch}). This has
historical reasons because homotopy colimits were first defined in categories where all 
object were cofibrant. 
\end{rema}

We apply these results to $\Mon$ and prove

\begin{theo}\label{6_15}
 The classifying space functor $$B:(\Mon,\Monw)\to (\top,\topw)$$ preserves homotopy colimits up to genuine homotopy equivalences.
More precisely, for any diagram $D:\scC\to \Mon$ the natural map
$$\hocolim^{\top}BD\to B(\hocolim^{\Mon}D)$$
is a homotopy equivalence.
\end{theo}

\begin{proof}
By definition of the homotopy colimit functor it suffices to prove the well-pointed case.

 Consider the diagram
$$\xymatrix{
(\Monw)^\mathcal{C} \ar[rr]^{\gamma_{(\Monw)^{\scC}}}\ar[d]_{ B^{\scC}} && \Ho(\Monw)^\mathcal{C} \ar[rr]^{\Ho\colim} \ar[d]_{\Ho B^{\scC}} 
&& \Ho\Mon \ar[d]^{\Ho B}
\\
(\topw)^\mathcal{C} \ar[rr]^{\gamma_{\topw}} && \Ho(\top)^\mathcal{C} \ar[rr]^{\Ho\colim}
&&
\Ho\topw
}
$$
and recall that $\Ho\colim$ is induced by the homotopy colimit functor. Since $B$ preserves weak equivalences in the
well-pointed case, $B^{\scC}$ induces $\Ho B^{\scC}$ so that the left square commutes up to natural equivalence.
The right square commutes up to natural equivalence, because
the corresponding square of right adjoints commutes. Hence, for any diagram $D$ in $\Mon$, the natural map
$$\hocolim^{\topw}BD\to B(\hocolim^{\Monw}D)$$
becomes an isomorphism in $\Ho\topw =\pi\topw$.
\end{proof}

%
%
%
%

\end{document}